%% file: hden.tex
\documentclass[11pt,a4paper,leqno]{article}
\usepackage[T1]{fontenc}
\usepackage{amsmath}
\usepackage{amssymb}
\usepackage{amsthm}
\usepackage{url}
\usepackage[noend]{algorithmic}
\usepackage{algorithm}

\numberwithin{equation}{section}

\newcommand{\intpart}[1]{\left\lfloor #1 \right\rfloor}
\newcommand{\suppress}[1]{}
\newcommand{\dsm}[1]{\mbox{$\displaystyle #1 $}}
\newcommand{\bigo}[1]{O\!\left(#1\right)}
\newcommand{\smallo}[1]{o\!\left(#1\right)}
\newcommand{\intf}[2]{\left[#1,\, #2\right]}

\def\unz{\underline{Z}}
\def\unb{\underline{B}}
\def\unc{\underline{C}}

\def\al{\alpha}
\def\de{\delta}
\def\De{\Delta}
\def\om{\omega}
\def\cd{\mathcal{D}}
\def\cm{\mathcal{M}}
\def\cn{\mathcal{N}}
\def\cp{\mathcal{P}}
\def\qh{\widehat{q}}
\def\veps{\varepsilon}
\newcommand{\R}{\mbox{$\mathbb{R}$}}
\newcommand{\C}{\mbox{$\mathbb{C}$}}
\newcommand{\N}{\mbox{$\mathbb{N}$}}
\newcommand{\qtx}[1]{\quad\text{#1}\quad}

\renewcommand{\d}{\,\mathrm{d}}
\newcommand{\Li}{\operatorname{Li}}
\def\LR{\Longrightarrow}

\newcommand{\abs}[1]{\left\vert#1\right\vert}
\newcommand{\set}[1]{\left\{#1\right\}}
\newcommand{\gfd}[2]{\genfrac{}{}{0pt}{2}{#1}{#2}}
\newcommand{\sumd}[2]{\sum_{\gfd{#1}{#2}}}
\newcommand{\Pm}{P^-}
\newcommand{\Pp}{P^+}

\newcommand{\intfg}[2]{\left[#1,\, #2\right)}
\newcommand{\oeis}{OEIS}
\newcommand{\ppn}[1]{{}^*\!#1}
\newtheorem{theorem}{Theorem}[section]
\newtheorem{definition}{Definition}[section]
\newtheorem{prop}{Proposition}[section]
\newtheorem{coro}{Corollary}[section]
\newtheorem{lem}{Lemma}[section]

\title{Maximal product of primes whose sum is bounded}
\author{Marc Del\'eglise and Jean-Louis Nicolas 
\footnote{Research partially supported by CNRS, Institut Camille Jordan, UMR 5208.}}

\begin{document}

\maketitle

\hfill
\begin{minipage}{0.5\linewidth}
To the memory of A.A. Karatsuba, on the occasion of his 75th anniversary.
\end{minipage} 
\medskip

\begin{abstract}
If $n$ is a positive integer, let $h(n)$ denote  
the maximal value of the product $q_1 q_2 \ldots q_j$ for all families of 
primes $q_1 < q_2 < \ldots < q_j$ such that  $q_1+q_2+\ldots+q_j \le n$. 
We shall give some properties of this function $h$ and describe an
algorithm able to compute $h(n)$ for any $n$ up to $10^{35}$. 
\end{abstract}

\def\refname{References}

\section{Introduction}\label{parint}

\subsection{Function $h(n)$}\label{parh}

If $n\geq 2$ is an integer, let us define $h(n)$ as the greatest
product of a family of primes $q_1 < q_2 < \ldots < q_j$ the sum of
which does not exceed $n$.

Let $\ell$ be the additive function such that  $\ell(p^\al)=p^\al$ for
$p$ prime and $\al \geq 1$. In other words, if the standard
factorization of $M$ into primes is
$M=q_1^{\al_1}q_2^{\al_2}\ldots q_j^{\al_j}$, we have 
$\ell(M)=q_1^{\al_1}+q_2^{\al_2}+\ldots +q_j^{\al_j}$ and $\ell(1)=0$. If $\mu$
denotes the Möbius function, $h(n)$ can also be defined as
\begin{equation}\label{h}
h(n)=\max_{\substack{\ell(M) \leq n\\\mu(M)\neq 0}} M.
\end{equation}
Note that 
\begin{equation}\label{lhn<=n}
\ell(h(n))\leq n.
\end{equation}
From the unicity of the factorization of $h(n)$ into primes, the
maximum in \eqref{h} is attained in only one point. It is convenient to set
\[
h(0)=h(1)=1.
\]
$(h(n))_{n \ge 1}$ is sequence A159685 of the \oeis\ (Online
Encyclopedy of Integer Sequences). A table of the $50$ first values of
$h(n)$ is given at the end of the paper.  A larger table may be found
on the authors's web sites \cite{DEL.WEB,NIC.WEB}.

In \cite{LAN}, Landau has introduced  the function $g(n)$ as the
maximal order of an element in the symmetric group ${\mathfrak S}_n$;
he has shown that
\begin{equation}\label{g}
g(n)=\max_{\ell(M) \leq n}M.
\end{equation}
The introductions of \cite{DNZ} and \cite{PPLUS} recall the main
properties of Landau's function $g(n)$ which is mentionned as entry
A002809 in \cite{SLOPLO}. From \eqref{h} and \eqref{g}, it follows
that
\begin{equation}\label{h<g}
h(n)\leq g(n),\quad (n\geq 0).
\end{equation}

In this article, we shall give some properties of $h(n)$ and describe
an algorithm able to calculate $h(n)$ for any $n$ up to $10^{35}$.

\subsection{Notation}

\begin{itemize}
\item
We denote by $\N$ the set of non-negative integers.
\item
The symbol $p$ will always denote a prime number.
\item
For every arithmetic function $f\ : \N \to \C$,
we define
\begin{equation}\label{defsfx}
\pi_{f}(x) = \sum_{p \le x,\ p \text{ prime}} f(p) 
\end{equation}
\item In particular, for $f(n) = 1$, we will note, as usual
$\pi(x) = \pi_{1}(x)$ the number of primes up to $x$.
\item 
For $f(n) = n$ we define
\begin{equation}\label{defpi1}
\pi_{id}(x) = \sum_{p \le x,\ p \text{ prime}} p 
\end{equation}
\item
We denote by $p_j$ the $j$-th prime and we set
$\sigma_0=0$, $N_0=1$ and, for $j \ge 1$,
\begin{equation}\label{Sj}
\sigma_j= \pi_{id}(p_j) = p_1+p_2+\ldots + p_j,\qquad N_j=p_1p_2\ldots p_j.  
\end{equation}
In \S\,\ref{parfires}, for all
$j\geq 1$, we shall prove that
$h(\sigma_j)=N_j.$
\item
If $m$ is an integer, we denote by $m^\star$ the smallest prime $p$
satisfying $p \geq m$ and, if $m \ge 2$, by $\ppn m$ the largest
prime $p$ satisfying $p \leq m$.
\item
$P^+(m)$ (resp. $P^-(m)$) will denote the largest (resp. smallest) prime
factor of $m \ge 2$. It is convenient to set 
$\Pp(1) = -\infty$, $\Pm(1) = +\infty$.
\item
$\omega(n)$ is the number of distinct prime factors of $n$
and $\Omega(n)$ the number of prime factors of $n$, counted
with multiplicity.
$\mu(n)$ is M\"obius's function.
\item
For $x > 1$, $\log_2(x) = \log\log x$.
\item
$\Li$ is the integral logarithm defined for $x > 1$ by
\[
\Li(x) = \lim_{\veps \to 0,\, \veps > 0}\
\int_{0}^{1-\veps} + \int_{1+\veps}^{x} \frac{\d t}{\log t}
=  \gamma + \log_2 x + \sum_{n \ge 1} \frac{(\log x)^n}{n n!}
\]
where $\gamma$ is Euler's constant.
\end{itemize}

\subsection{Functions $h_j(n)$}\label{parhj}

For $n \geq 0$, let $k=k(n)$ be the non-negative integer defined by
\begin{equation}\label{k}
\sigma_k = \pi_{id}(p_k) \leq n < \pi_{id}(p_{k+1}) = \sigma_{k+1}.
\end{equation}
It is the maximal number of prime factors of $h(n)$. For $0\leq j \leq
k=k(n)$, let us set
\begin{equation}\label{hj}
h_j(n)=\max_{\substack{\ell(M) \leq n\\\mu(M)\neq 0,\;\om(M)=j}} M
\end{equation}
where $\om(M)$ is the number of prime factors of $M$. 
For $n \ge 0$, we have
\begin{equation}\label{h0}
h_0(n) = 1
\end{equation}
while, for $n \ge 2$, we have
\begin{equation}\label{h1}
h_1(n) = \ppn n \ge 2.
\end{equation}
Note that 
\begin{equation}\label{lhjn<=n}
\ell(h_j(n))\leq n.
\end{equation}
In \S\,\ref{parinchj}, we
prove that, for all $n$'s, the sequence $h_j(n)$ is increasing
on $j$, so that
\begin{equation}\label{h=hk}
h(n)=h_k(n),\quad (n\geq 0).
\end{equation}
Our proof is not that simple. A possible reason is that this
increasingness relies on the properties of the whole set of primes
$\cp$. Let $\cp'$ be a subset of $\cp$ and $\cn_{\cp'}$ the set of
integers whose prime factors belong to $\cp'$. We may consider
\begin{equation}\label{hjP'}
h_j(n,\cp')=\max_{\substack{M\in \cn_{\cp'},\;\ell(M) \leq
    n\\\mu(M)\neq 0,\;\om(M)=j}} M\ .
\end{equation}
By choosing $\cp'=\{2,3,11,13,17,19,23,\ldots\}=\cp \setminus \{5,7\}$, 
we observe that 
\[
h_2(24,\cp')=11\cdot 13 =143 \;\; > \;\; h_3(24,\cp')=2\cdot 3\cdot
19=114.
\]

In \S\,\ref{parbouhj}, we give an upper bound for $h_j(n)$
 which will be useful in \S\,\ref{parinchj} where our proof of the
 increasingness of $h_j$ is given. In \eqref{hj}, $h_j(n)$ can be
 considered as the solution of a problem of optimization with prime
 variables. The upper bound of $h_j(n)$ is obtained by relaxing some
 constraints so that certain variables are no longer primes, but only integers.

\subsection{Elementary computation of $h(n)$ and $h_j(n)$}\label{parcomph}

The naive algorithm described in \cite{DNZ} to
compute $g(n)$ can be easily adapted to calculate $h(n)$ for $1\leq n
\leq N$. Note that, for the prime factors of $h(n)$, Corollary 
\ref{coropp'h} below furnishes the upper bound
\begin{equation*}
\Pp(h(n)) \le p_{k(n)+1} + p_{k(n)+2}.
\end{equation*}
It also can be adapted to compute $h_j(n)$. For $r\geq j \geq 1$ and
$n\geq \sigma_j$, let us define
\[
h_j^{(r)}(n)=\max_{\substack{P^+(M)\leq p_r,\;\ell(M) \leq n\\
\mu(M)\neq 0,\;\om(M)=j}} M\ .
\]
We have the induction relation
\[
h_j^{(r+1)}(n)=\max(h_j^{(r)}(n),p_{r+1} h_{j-1}^{(r)}(n-p_{r+1})).
\]
Indeed, either $p_{r+1}$ does not divide $h_j^{(r+1)}(n)$, and 
$h_j^{(r+1)}(n) = h_j^{(r)}(n)$ holds, or $p_{r+1}$ divides
$h_j^{(r+1)}(n)$, and $h_j^{(r+1)}(n)=p_{r+1}
h_{j-1}^{(r)}(n-p_{r+1})$, which implies $n\geq p_{r+1}+\sigma_{j-1}$.

Moreover, if $p_r\geq n$, we have $h_j^{(r)}(n)=h_j(n)$,
$h_r^{(r)}(n)=N_r$ and $h_1^{(r)}(n)=\ppn n$ for $n < p_r$ while,
for $n\geq p_r$, $h_1^{(r)}(n)=p_r$ holds. So, we may write
{\bf algorithm \ref{algohj}}, which has been used to
calculate the table in appendix. The merging and pruning method
described in \cite[\S 2.2]{DNZ} can be used to improve the running
time.

\renewcommand\algorithmicindent{2em}

\begin{algorithm}
\caption{Computation of $h_j(n)$ for $2 \le n \le nmax$ and $1 \le j
  \le k(n)$} 
\label{algohj}
\begin{algorithmic}
\STATE{{\bf Procedure} ComputeHj(nmax)}
\STATE{$r = 1;\ p = p_r;\ kmax = k(nmax);\ pmax = p_{kmax+1}+p_{kmax+2}$}
\WHILE {$p \le pmax$}
\FOR{$n$ from  $\sigma_r$ to $nmax$}{
\STATE{$H[r,n] = N_r$}}\;
\ENDFOR
\STATE{$jmax = \min(r-1,kmax)$}
\FOR{$j$ from $jmax$ by $-1$ to $2$}
\FOR{$n$ from $nmax$ by $-1$ to $p+\sigma_{j-1}$}
\STATE{$H[j,n] = \max(H[j,n],p*H[j-1,n-p])$}
\ENDFOR
\ENDFOR
\FOR{$n$ from $p$ to $nmax$}
\STATE{$H[1,n] = p$;}
\ENDFOR
\STATE{$r = r+1;\ p = p_r$}
\ENDWHILE
\end{algorithmic}
\end{algorithm}

In \S\, \ref{paralgoh}, a more sophisticated algorithm to calculate
$h(n)$ is given. It is based on a fast method to compute
$\pi_{id}(x)$, which is explained in \S\, \ref{parpifcomp}.

\section{Some lemmas}\label{lemmas}

\begin{lem}\label{lem11/8}
If $m\geq 2$ is an integer, let us denote by $m^\star$ (resp. ${
}^\star m$) the smallest 
(resp. largest) prime $p$ satisfying $p\geq m$ (resp. $p\leq m)$. Then
\[
m^\star \leq \frac{11}{8}m \quad \text{ and } \quad \ppn m \geq
\frac{7}{10} m
\]
hold. 
\end{lem}

\begin{proof}
We use the result of \cite{DUS10}: for $x\geq 396738$, the interval 
$[x,x+\frac{x}{25 \log^2 x}]$ contains a prime number. As $396833$ is prime, 
we deduce that, for $p_i \geq 396833$,
\begin{equation}\label{lem11/81}
\frac{p_{i+1}}{p_i}\leq 1+\frac{1}{25 \log^2 p_i}
\leq 1+ \frac{1}{25 \log^2 396833} < 1.00025 <\frac{11}{8} < \frac{10}{7} .
\end{equation} 

If $m$ is prime, $m^\star=\ppn m=m$ holds, while, if $m$ is not prime, 
we define $p_i$ by $p_i < m < p_{i+1}$; we have 
$m^\star=p_{i+1}$, $ \ppn m=p_i$,
\[
\frac{m^\star}{m}\leq \frac{p_{i+1}}{p_i+1} <  \frac{p_{i+1}}{p_i},
\qquad
\frac{\ppn m}{m} \geq \frac{p_i}{p_{i+1}-1} > \frac{p_i}{p_{i+1}}
\]
and, if $p_i\geq 396833$,  the result follows from
\eqref{lem11/81}. Finally, it remains to check that 
\[
\frac{p_{i+1}}{p_i+1} \leq \frac{11}{8} \quad \text{ and } \quad
\frac{p_i}{p_{i+1}-1} \geq \frac{7}{10}
\]
hold for all $p_i$'s satisfying $2 \leq p_i < 396833$.
\end{proof}

\begin{lem}\label{lemp+p'}
Let $p < p'$ be two primes. There exists a third prime $p''$
satisfying
\begin{equation}\label{ineqp''}
p+p' \leq p'' \leq pp'-p+1.
\end{equation}
\end{lem}

\begin{proof}
Let us show that $p''=(p+p')^\star$ satisfies \eqref{ineqp''}. By Lemma
\ref{lem11/8}, it suffices to prove that $\frac{11}{8} (p+p') \leq
pp'-p+1$, i.e:
\begin{equation}\label{pp'<}
pp'\left( 8-\frac{11}{p}-\frac{19}{p'} + \frac{8}{pp'}\right) \geq 0.
\end{equation}
If $p\geq 3$ and $p'\geq 5$, we have $\frac{11}{p}+\frac{19}{p'}\leq
\frac{11}{3}+\frac{19}{5} < 8$ and \eqref{pp'<} holds. Similarly, if
$p=2$ and $p' \geq 11$, the inequality $\frac{11}{p}+\frac{19}{p'} \leq
\frac{11}{2}+\frac{19}{11} < 8$ implies \eqref{pp'<}. In the three
remaining cases, $p=2$ and
$p'\in \{3,5,7\}$, it is easy to check that $p''=(p+p')^\star$ satisfies \eqref{ineqp''}.
\end{proof}

\begin{lem}\label{lemp+p'5/6}
Let $p$ and $p'$ be two prime numbers satisfying $3 \leq p <
p'$ and $pp' \ne 15$. There exists a prime $p''$ such that
\begin{equation}\label{ineqp''5/6}
p+p' \leq p'' \leq \frac 56 pp'-p.
\end{equation}
\end{lem}

\begin{proof}
The proof is similar to the one of the preceding lemma. From
Lemma \ref{lem11/8}, to show that $p''=(p+p')^\star$ satisfies
\eqref{ineqp''5/6}, it suffices to show that $\frac{11}{8}(p+p')\leq
\frac 56 pp'-p$, i.e. $33/p+57/p' \leq 20$, which evidently holds for
$p\geq 3$ and $p'\geq 7$.
\end{proof}

\begin{lem}\label{lempi+pi}
For all $i \geq 2$, the following inequality
\begin{equation}\label{pipi1}
p_i+p_{i-1} \leq p_{2i-1}
\end{equation}
holds. Moreover, let $b$ be a positive integer; there exists a
positive integer $i_0=i_0(b)$ such that we have
\begin{equation}\label{pipib}
p_i+p_{i-1} <  p_{2i-b} \quad \text{ for } i\geq i_0(b).
\end{equation}
The table below gives some values of $i_0(b)$
\[
\begin{array}{|r|ccccccccccccccc|}
\hline
b=&1&2&3&4&5&6&7&8&9&10&12&13&18&30&3675\\
\hline
i_0=&3&4&7&8&18&19&27&28&36&39&50&53&85&149&33127\\
\hline
\end{array}
\]
\end{lem}

\begin{proof}
We start from the two inequalities
\begin{equation}\label{lip1}
p_i \leq i(\log i + \log \log i -\al), \quad (\al=0.9484,\;i\geq 39017),
\end{equation}
\begin{equation}\label{lip2}
p_i \geq i(\log i + \log \log i -1), \quad (i\geq 2)
\end{equation}
which can be found in \cite{DUSMC}. From \eqref{lip1}, it follows that
\begin{equation}\label{lip3}
p_{i-1}+p_i \leq (2i-1)(\log i + \log \log i -\al), \quad (i\geq 39018)
\end{equation}
while, if $i\geq \max(2,b)$, which implies $2i-b \ge 2$ and
$i \ge b$, \eqref{lip2} gives
\begin{equation}\label{lip5}
p_{2i-b}\geq (2i-b)(\log i +\log 2+\log\left(\frac{2i-b}{2i}\right)+ \log \log i -1).
\end{equation}
By using the inequality $\log t \leq t-1$, we get
\[
\log\left(\frac{2i-b}{2i}\right) =-\log\left(\frac{2i}{2i-b}\right)
\geq -\left(\frac{2i}{2i-b}-1\right) = -\frac{b}{2i-b}
\]
and \eqref{lip5} yields
\begin{equation}\label{lip6}
p_{2i-b}\geq (2i-b)(\log i + \log \log i +\log 2 -1) -b.
\end{equation}
Under the condition 
\begin{equation}\label{lip7}
i\geq \max(39018,b),
\end{equation}
the substraction of \eqref{lip3} from \eqref{lip6} gives
\begin{eqnarray}\label{lip8}
p_{2i-b}-p_{i-1}-p_i &\geq & (\log i+\log \log i+ \log 2)(1-b)\notag\\
& & \quad +\ 2i(\log 2-1+\al) -\log 2 -\al\notag\\
> \;\; (\log i\hspace{-3mm}&+&\hspace{-3mm}\log \log i+ \log 2)
\left[ \frac{1.283 \;i-1.642}{\log i+\log \log i+ \log 2}-(b-1)\right].
\end{eqnarray}
Now, the two functions $t\mapsto t/(\log t + \log \log t + \log 2)$
and $t\mapsto -1/(\log t + \log \log t + \log 2)$ are increasing for
$t\geq e^2$; choosing $i_1=39018$  and 
\begin{equation}\label{lip9}
b=\left\lfloor 1+ \frac{1.283 \;i_1-1.642}{\log i_1+\log \log i_1+
    \log 2}\right\rfloor =\lfloor3675.52\ldots\rfloor=3675
\end{equation}
shows that, for $i\geq i_1$, \eqref{lip7} is satisfied and that in
\eqref{lip8}, the bracket is positive. Therefore, 
\eqref{lip8} proves $p_i+p_{i-1} < p_{2i-3675}$ for $i\geq i_1=39018$.

To determine
the  entries of the table, for all $i$'s up to $39018$, we have
calculated $b_i=2i-1-\pi(p_i+p_{i-1})$ which is 
the smallest integer  such that $p_{i-1}+p_i <
p_{2i-b_i}$. Further, for each $b$ in the table, we have determined
$i_0(b)$ which is
the smallest integer $i_0$ such that, for $i_0(b)\leq i \leq 39018$, $b_i
\geq b$ holds.

As $i_0(1)=3$, for all $i\geq 3$, $p_i+p_{i-1} < p_{2i-1}$ holds. So,
\eqref{pipi1} follows from $p_2+p_1=3+2=5=p_3$.
\end{proof}

\renewcommand{\d}{\,\mathrm{d}}
\newcommand{\dd}[1]{\dfrac{\d}{\d #1}}

\begin{lem}\label{approxpi1}
Under Riemann hypothesis, for all $x \ge 41$ we have
\begin{equation}\label{equapprox}
\abs{\pi_{id}(x) - \Li(x^2)} \le \frac{5}{24\pi}x^{3/2} \log x.
\end{equation}
\end{lem}

\begin{proof}
Let us define $r(x)$ by $\pi(x) = \Li(x) + r(x)$
and assume the Riemann hypothesis. Then cf. \cite[(6.18)]{SCH76}~:
\begin{equation}\label{riemh}
\abs{r(x)} = \abs{\pi(x)-\Li(x)} \le\frac{1}{8\pi}\sqrt{x}\log x
\quad (\text{for } x \ge 2657).
\end{equation}

Let us denote $x_0=2657$. Then, from \eqref{defpi1}, Stieltjes's
integral gives~:
\begin{eqnarray*}
\pi_{id}(x) &=& \pi_{id}(x_0) + \int_{x_0^-}^{x} t\, \d [\pi(t)]\\
&=& \pi_{id}(x_0) + \int_{x_0}^{x} t \d (\Li(t))  + \int_{x_0^-}^{x} t \d [r(t)]\\
&=& \pi_{id}(x_0) + \Li(x^2)-\Li(x_0^2)
+ \left. t r(t)\right|_{x_0}^{x} - \int_{x_0}^{x} r(t) \d t.
\end{eqnarray*}
With \eqref{riemh}, it comes
\begin{eqnarray*}
\abs{\pi_{id}(x) -\Li(x^2)} &\le& \abs{\pi_{id}(x_0)-\Li(x_0^2) -x_0 r(x_0)}
+ \frac{x^{3/2}\log x}{8\pi} + \int_{x_0}^{x} \frac{\sqrt t \log
  t}{8\pi } \d t
\end{eqnarray*}
and, using
\dsm{
\int \sqrt t \log t = \frac{2}{3}t^{3/2} \left(\log t - \frac{2}{3} \right),
}
\begin{multline}\label{majpi1minLi2}
\abs{\pi_{id}(x) -\Li(x^2)} 
\le \frac{5}{24\pi}x^{3/2}\log x - \frac{1}{18\pi} x^{3/2}\\
+ \abs{\pi_{id}(x_0)-\Li(x_0^2) -x_0 r(x_0)} 
- \frac{1}{12\pi}x_0^{3/2}\log x_0 + \ \frac{1}{18\pi} x_0^{3/2}.
\end{multline}
The computation of
\begin{eqnarray*}
r(x_0) &=& \pi(x_0) - \Li(x_0) = 384 -399.59681\ldots = -15.59681\ldots \\
\pi_{id}(x_0) - \Li(x_0^2) &=& 464\,653 -480610.2863\ldots = -15957.2863\ldots
\end{eqnarray*}
and \eqref{majpi1minLi2} imply for $x \ge x_0$,
\[
\abs{\pi_{id}(x) -\Li(x^2)} \le \frac{5}{24\pi}x^{3/2}\log x
 - \frac{1}{18\pi} x^{3/2} -740.023\ldots
\le \frac{5}{24\pi}x^{3/2}\log x.
\]
which proves \eqref{equapprox} for $x \ge x_0 = 2657$. It
remains to check \eqref{equapprox} for \dsm{41 \le x \le 2657};
by setting 
\[
f_1(x) = \Li(x^2) - \frac{5}{24\pi} x^{3/2}\log x,
\quad
f_2(x) = \Li(x^2) + \frac{5}{24\pi} x^{3/2}\log x,
\]
it is equivalent to check
\begin{equation}\label{star}
f_1(x) \le \pi_{id}(x) \le f_2(x) 
\end{equation}
for $41 \le x \le 2657$.
One remarks that $f_1$ and $f_2$ are increasing for $x \ge 2$.
Therefore, to prove \eqref{star}, it suffices to check that
for every prime $p$ satisfying $41 \le p \le 2657$ we have
$f_1(p') \le \pi_{id}(x) \le f_2(p)$ where $p'$ is the prime
following $p$.
\end{proof}
\medskip

Note that, in the range $[2..2657]$, $\pi_{id}(x) - \Li(x^2)$
has several changes of sign, the smallest one being for
$x = 110.35\dots$
\medskip


\newcommand{\PRIME}{\operatorname{prime}}
\newcommand{\PI}{\operatorname{pi}}
\newcommand{\PIFTAB}{\operatorname{piftab}}

\begin{lem}\label{speedup}
Let $z$ and $u$ be two real numbers satisfying
$z \ge 1$ and $\sqrt{z} \le u \le z$. 
Suppose that we have precomputed the tables $\PRIME$,
$\PIFTAB$ and $\PI$.
The first two tables are indexed by the integers
$k$, $0 \le k \le \pi(u)$, and the third one
by the integers $t$, $0 \le t \le u$.
\begin{itemize}
\item
$\PRIME[k]$ contains $p_k$ ($p_0 = 1$).
\item
$\PIFTAB[k]$ contains $\pi_f(p_k)$.
\item
$\PI[t]$ contains $\pi(t)$.
\end{itemize}
\medskip

\noindent
Then the sum
\begin{equation}\label{sumzn}
\sum_{\sqrt z < q \le u,\ q \text{ prime}}  f(q) \pi_f\left(\frac{z}{q}\right)
\end{equation}
may be computed in $\bigo{\sqrt z/\log z}$ time.
\end{lem}

\begin{proof}
For $q > \sqrt z$, $z/q$ belongs to $\intfg{1}{\sqrt z}.$ The number
of primes in this interval is $\bigo{\sqrt z/\log z}$, thus the number
of values of $\pi_f(z/q)$ is $\bigo{\sqrt z/\log z}$.  We group the
$q$'s for which $\pi_f(z/q)$ takes the same value. Algorithm
\ref{algospeed} carries out this computation.

\begin{algorithm}
\caption{: Computation of the sum \eqref{sumzn} in
  \dsm{\bigo{\sqrt{z}/\log z}} time}
\label{algospeed}
\begin{algorithmic}
\STATE{$S = 0; \ imin = 1+\PI[\intpart{\sqrt{z}}]$}\\
\WHILE {$imin\le \PI[u]$}
\STATE{$q = \PRIME[imin]$}
\STATE{$s = \PI[z/q]$}
\STATE{$imax = \min(\PI(z/\PRIME[s]),\PI[u])$}
\STATE{$S = S + (\PIFTAB[imax]-\PIFTAB[imin-1])*\PIFTAB[s]$}
\STATE{$imin = imax+1$}
\ENDWHILE
\RETURN{$S$}
\end{algorithmic}
\end{algorithm}

Let us give some words to convince of the correctness of algorithm 2~:
let us note $s = \pi(z/q)$.  Then $p_s$ is the largest prime $\le z/q$.
For $q'$ prime, $q' \ge q$, we have
$\pi_f(z/q') = \pi_f(p_s) = \PIFTAB[s]$ if and only if
$z/q' \ge p_s$ i.e. $q' \le z/p_s$, in other terms, $\pi(q') \le
\pi(z/p_s)$.
Thus the largest prime $q'$ in the range $[q..u]$ such that
$\pi_f(z/q') = \pi_{f}(p_s)$ is $p_i$ where 
$i = \min(\pi(z/p_s),\pi(u))$. 
\end{proof}


\section{First results}\label{parfires}

\begin{prop}\label{propSj}
Let $j$ be a positive integer and $\sigma_j$ and $N_j$ be defined by
\eqref{Sj}. We have
\[
h(\sigma_j)=N_j.
\]
\end{prop}

\begin{proof}
It is easy to see that $h(\sigma_1)=h(2)=2=N_1$ and
$h(\sigma_2)=h(5)=6=N_2$. Now, we
may suppose that $j\geq 3$, i.e. $p_j\geq 5$ and we set
$\rho=p_j/\log p_j$. The function $t\mapsto t/\log t$ is increasing
for $t\geq e$ and, since $2/\log 2 < 5/\log 5$,
we have, for $1\leq i < j$, $p_i/\log p_i < \rho$ and for
$i > j$, $p_i/\log p_i > \rho$; in other words, $i-j$ and $p_i/\log
p_i-\rho$ have the same sign.

Let $M$ be a product of $r$ distinct primes, $M=Q_1Q_2\ldots Q_r$,
with $r\geq 0$. After a possible simplification by $s$ primes 
($0\leq s \leq \min(j,r)$), we may write
\[
\frac{M}{N_j}=\frac{p_{j_1}p_{j_2}\ldots p_{j_u}}{p_{k_1}p_{k_2}\ldots
  p_{k_v}}
\]
with $u=r-s$, $v=j-s$ and
\[
p_{k_1} < p_{k_2} < \ldots < p_{k_v} \leq p_j <p_{j_1} < p_{j_2} <
\ldots <  p_{j_u}.
\]
Let $f(M)=\ell(M)-\rho \log M$. From the definition of $\ell$, the
function $f$ is additive and we have
\begin{equation}\label{p11}
f(M)-f(N_j)=\sum_{i=1}^u (p_{j_i}-\rho \log p_{j_i})-\sum_{i=1}^v
(p_{k_i}-\rho \log p_{k_i}) \geq 0
\end{equation}
since each term of the first sum is non-negative while, in the second
sum, each term is non-positive.

From \eqref{h}, since $\ell(N_j)=\sigma_j$, in order to prove that
$h(\sigma_j)=N_j$, we must show that, for all squarefree number $M$
satisfying $\ell(M)\leq \sigma_j=\ell(N_j)$, we have $M\leq N_j$. But,
for such an $M$, \eqref{p11} yields
\[
f(M)=\ell(M)-\rho \log M  \geq f(N_j)=\ell(N_j)-\rho \log N_j=\sigma_j
-\rho \log N_j
\]
whence
\[
\frac{M}{N_j}\leq \exp\left(\frac{\ell(M)-\sigma_j}{\rho}\right)\leq
1,
\]
which completes the proof of Proposition \ref{propSj}.
\end{proof}

\begin{prop}\label{propSjSr}
Let $r$ and $j$ be two positive integers and $\sigma_j$, $N_j$ and $h_j$ be defined by
\eqref{Sj} and \eqref{hj}. We have
\begin{equation}\label{hjSj+r}
h_j(\sigma_{j+r}-\sigma_r)=N_{j+r}/N_r=p_{r+1}p_{r+2}\ldots p_{r+j}.
\end{equation}
Moreover, if $n\geq \sigma_{j+r}-\sigma_r$ we have
\begin{equation}\label{lhjn>=}
\ell(h_j(n)) \geq \sigma_{j+r}-\sigma_r.
\end{equation}
\end{prop}

\begin{proof}
The proof is similar to the one of Proposition \ref{propSj}. Let us set
\begin{equation}\label{rho}
\rho=\frac{p_{j+r}-p_r}{\log (p_{j+r}/p_r)}\qquad \text{ and } \qquad
\rho'=\rho\log p_r-p_r.
\end{equation}
Since, for $t\neq 1$, $(t-1)/t < \log t < t-1$ holds, we have $p_{j+r}
>\rho > p_r \geq 2.$
For a squarefree number $M$, we consider the additive function
\[
f(M)=\ell(M)-\rho \log M + \rho' \om(M)=\sum_{p\mid M} f(P)=
\sum_{p\mid M} (p-\rho \log
p+\rho').
\]
We will prove that $f$ attains its minimum in $N=N_{j+r}/N_r$. From
\eqref{rho}, it follows that $f(p_{j+r})=f(p_r)=0$ and the study of
the function $t\mapsto t-\rho\log t+\rho'$ shows that
\[
f(p)
\begin{cases}
> 0 & \text{ for } p < p_r \text{ or } p > p_{j+r}\\
< 0 & \text{ for } p_r < p < p_{j+r}\\
= 0 & \text{ for } p=p_r \text{ or } p=p_{j+r}.
\end{cases}
\]
Therefore, we have
\begin{equation}\label{prop21}
f(M)-f(N)=
\sum_{\substack{p\mid M\\p \;<\; p_r \text{ or } p\; >\;p_{j+r}}} f(p)
-\sum_{\substack{p\nmid M\\ p_r \;< \;p \;< \;p_{j+r}}} f(p) \geq 0.
\end{equation}
From \eqref{hj}, we have to show that, for any squarefree integer $M$
satisfying $\ell(M)\leq \sigma_{j+r}-\sigma_r=\ell(N)$ and $\om(M)=j=\om(N)$, we
have $M\leq N$. For such an $M$, \eqref{prop21} gives
\[
\ell(M)-\rho\log M+\rho'\om(M) \geq \ell(N)-\rho\log
N+\rho'\om(N)
\]
yielding
\[
\frac{M}{N}\leq \exp\left(\frac{\ell(M)-\ell(N)}{\rho}\right)\leq 1,
\]
which, together with $\ell(N)=\sigma_{j+r}-\sigma_r$, proves \eqref{hjSj+r}.

To prove \eqref{lhjn>=}, first, from \eqref{hj}, we observe that
$h_j(n) \geq N=N_{j+r}/N_r$. Setting $M=h_j(n)$ in \eqref{prop21} and
noting that $\om(M)=\om(N)=j$, we see that
\[
\ell(M) \geq \ell(N) +\rho \log \frac MN \geq
\ell(N)=\sigma_{j+r}-\sigma_r
\]
which proves \eqref{lhjn>=}.
\end{proof}

\begin{prop}\label{proppp'h}
Let $n\geq 2$ be an integer and $p < p'$ two prime numbers 
which do not divide $h(n)$. Then the largest prime divisor $P^+(h(n))$ of
$h(n)$ satisfies
\[
P^+(h(n)) < p+p'.
\]
\end{prop}

\begin{proof}
Let us assume that the set of prime factors of $h(n)$ not smaller 
than $p+p'$ is not empty and let $q \geq p+p'$ be its smallest element.
\begin{itemize}
\item
If $q < pp'$, by setting $M=\frac{pp'}{q} h(n)$, we have by \eqref{lhn<=n}
\[
\ell(M)=p+p'-q+\ell(h(n))\leq \ell(h(n))\leq n
\]
and thus, from \eqref{h},
\begin{equation}\label{proppp'h1}
h(n) \geq M=\frac{pp'}{q} h(n),
\end{equation}
in contradiction with $q < pp'$.

\item
If $q > pp'$, i.e. $q\geq pp'+1$, by Lemma \ref{lemp+p'}, 
the interval $[p+p',q-p]$ contains a prime; thus
the prime $p''=\ppn(q-p)$ satisfies $p+p'\leq p'' \leq q-p < q$ and,
from the definition of $q$, $p''$ does not divide $h(n)$. 
By Lemma \ref{lem11/8}, $p'' =\ppn (q-p)\geq \frac{7}{10}(q-p)$ holds, whence
\[
q \leq \frac{10}{7} p''+p=
\frac{pp''}{7}\left(\frac{10}{p}+\frac{7}{p''}\right).
\]
We have  $p\geq 2$,  $p' \geq 3$ and $p'' \geq p+p' \geq 5$, so that 
$\frac{10}{p}+\frac{7}{p''} \leq \frac{10}{2}+\frac{7}{5} < 7$, 
yielding $q < pp''$. By considering 
$M=\frac{pp''}{q} h(n)$, as in \eqref{proppp'h1}, we get
\[
h(n) \geq M=\frac{pp''}{q} h(n) > h(n),
\]
a contradiction.
\end{itemize}
\end{proof}

\begin{coro}\label{coropp'h}
If $k=k(n)$ is defined by \eqref{k}, the largest prime factor 
of $h(n)$ satisfies
\[
P^+(h(n)) < p_{k+1}+p_{k+2}.
\]
\end{coro}

\begin{proof}
The number of prime factors of $h(n)$ does not exceed $k$, so that, 
among $p_1,p_2,\ldots, p_{k+2}$ there are certainly two 
prime numbers $p$ and $p'$ not dividing $h(n)$. By applying Proposition \ref{proppp'h}, 
we get $P^+(h(n)) < p+p' \leq p_{k+1}+p_{k+2}$.
\end{proof}

\begin{prop}\label{proppp'hj}
Let $n\geq 5$ be an integer, $k\geq 2$ be defined by \eqref{k} and $j$
an integer satisfying $2\leq j\leq k$. Let us supose that there
exists  two prime numbers, $p, p'$ not dividing $h_{j-1}(n)$, and 
satisfying $3\leq p < p'$ and
$P^+(h_{j-1}(n))\geq p+p'$ where $P^+(h_{j-1}(n))$ is the largest prime divisor of
$h_{j-1}(n)$. Then the inequality
\[
h_{j}(n) > \frac 65 h_{j-1}(n)
\]
holds.
\end{prop}

\begin{proof}
Let us consider two cases~:
\paragraph{Case 1~: $pp' > 15$.}
Let us denote by $q \leq P^+(h_{j-1}(n))$  the smallest prime dividing
$h_{j-1}(n)$ and satisfying $p+p' \leq q$.

\begin{itemize}
\item
If $q < \dfrac 56 pp'$, we set $M=\dfrac{pp'}{q} h_{j-1}(n)$; we have
$\om(M)=j$ and
$\ell(M)=p+p'-q+\ell(h_{j-1}(n))\leq \ell(h_{j-1}(n))$ so that, from
\eqref{lhjn<=n}, $\ell(M) \leq n$ holds and \eqref{hj} yields
\begin{equation}\label{proppp'hj1}
h_{j}(n) \geq M > \frac 65 h_{j-1}(n)
\end{equation}
as required.
\item
If $q \geq \dfrac 56 pp'$, we set $p''=\ppn (q-p)$; from 
Lemma \ref{lemp+p'5/6},
$p+p' \leq p'' \leq q-p < q$ holds,  and, from the definition of $q$,
$p''$ does not divide $h_{j-1}(n)$.

By Lemma \ref{lem11/8}, we get $p'' =\ppn (q-p)\geq
\frac{7}{10}(q-p)$, which implies
\[
q \leq \frac{10}{7} p''+p=
\frac{pp''}{7}\left(\frac{10}{p}+\frac{7}{p''}\right).
\]
But $p\geq 3$,  $p' \geq 7$, $p'' \geq p+p' \geq 10$, 
thus $p'' \geq 11$, and
$\frac{10}{p}+\frac{7}{p''} \leq \frac{10}{3}+\frac{7}{11} < \frac{35}{6}$, 
yielding $q < \frac 56 pp''$. By setting
$M=\frac{pp''}{q} h_{j-1}(n)$, as in \eqref{proppp'hj1}, we get
$h_{j}(n) \geq M > \frac 65 h_{j-1}(n)$.
\end{itemize}

\paragraph{Case 2~: $p=3$,  $p'=5$.}
\begin{itemize}
\item
If $\Pp(h_{j-1}(n)) \le 13$, which implies $n \le \pi_{id}(13) = 41$,
examining the table of Fig. 1 shows that, for $n \le 41$, we have
$h_j(n) \ge \frac{6}{5} h_{j-1}(n)$ with equality if and only if
$h_{j-1}(n) = 5$, $35$, $385$ or $5005$.

\item
If $\Pp(h_{j-1}(n)) \ge 17$, and $11$ does not divide
$\Pp(h_{j-1}(n))$,
then we apply case 1 with $p=3$, $p'=11$, while, if $11$
divides $\Pp(h_{j-1}(n))$,
\dsm{h_j(n) \ge \frac{3 \cdot 5}{11} h_{j-1}(n) 
> \frac{6}{5} h_{j-1}(n)} holds.
\end{itemize}
\end{proof}

\section{Bounding $h_j(n)$}\label{parbouhj}

\begin{prop}\label{propbouhj}
Let $j\geq 1$ and $n\geq \sigma_j$ (where $\sigma_j$ has been introduced 
in \eqref{Sj}) be two integers; we define $r\geq 0$ by
\begin{equation}\label{encn}
\sigma_{j+r}-\sigma_r \leq n < \sigma_{j+r+1}-\sigma_{r+1}
\end{equation}
and $n'$ by
\begin{equation}\label{n'}
0\leq n'=n-(\sigma_{j+r}-\sigma_r)  < p_{j+r+1}-p_{r+1}.
\end{equation}
Then we have
\begin{equation}\label{hj<=}
h_j(n) \leq  p_{r+1}p_{r+2}\ldots p_{r+j}
\frac{p_{j+r+1}}{p_{j+r+1}-n'}= \frac{N_{j+r+1}}{N_r (p_{j+r+1}-n')}\cdot
\end{equation}
\end{prop}

\begin{proof}
From its definition \eqref{hj}, $h_j(n)$ is a product of $j$
primes. Let us denote by $A_1,A_2,\ldots,A_u$ (with $0 \leq u \leq j$) its
prime factors exceeding $p_{j+r+1}$ and by $B_1,B_2, \ldots B_{r+1+u}$
the primes  $\leq p_{j+r+1}$ and not
dividing $h_j(n)$; we have
\begin{equation}\label{hjAB}
h_j(n)=\frac{N_{j+r+1} A_1 A_2 \ldots A_u}{B_1 B_2 \ldots B_{r+1+u}}
\end{equation}
(where the product $A_1 A_2 \ldots A_u$ should be replaced by $1$ when
$u=0$) and
\begin{equation}\label{B}
2\leq B_1 < \ldots < B_{r+1+u} \leq p_{j+r+1} < p_{j+r+2} \leq A_1 <
\ldots < A_u.
\end{equation}
Further, let us introduce $\nu=\ell(h_j(n))$; by \eqref{lhjn<=n} 
and \eqref{encn}, we have 
\begin{equation}\label{nu<=n}
\nu \leq n < \sigma_{j+r+1}-\sigma_{r+1}
\end{equation}
and it follows from Proposition \ref{propSjSr}, \eqref{lhjn>=},  that 
\begin{equation}\label{nu>=}
\nu=\ell(h_j(n)) \geq \sigma_{j+r}-\sigma_r .
\end{equation}
Moreover, \eqref{hjAB} implies
\begin{equation}\label{nu}
\nu=\ell(h_j(n)= \sigma_{j+r+1}-\sigma_{r} + \sum_{i=1}^u
(A_i-B_{r+1+i})-\sum_{i=1}^r (B_i-p_i) -B_{r+1}.
\end{equation}

Now, we consider the optimization problem (where $\nu, r,u,
A_1,A_2,\ldots, A_u$ are fixed)
\begin{equation}\label{opt}
\cm=\max_{\unz\in \cd} \frac{A_1 A_2 \ldots A_u}{f(\unz)}
\end{equation}
where $\cd$ is a subset of $\N^{r+1+u}$, $\unz=(Z_1,Z_2,\ldots,
Z_{r+1+u})$,
\[
f(\unz)=Z_1Z_2\ldots Z_{r+1+u}
\]
and the set $\cd$ is defined by
\begin{equation}\label{Zipi}
Z_i \geq p_i, \quad (1\leq i \leq r+1),
\end{equation}
\begin{equation}\label{Zirp1}
Z_i  <  Z_{r+1}, \quad (1\leq i \leq r),
\end{equation}
\begin{equation}\label{ZiAi}
Z_{r+1} < Z_{r+1+i} \leq A_i , \quad (1\leq i \leq u)
\end{equation}
and
\begin{equation}\label{URnu}
U(\unz)-R(\unz)-Z_{r+1}+\sigma_{j+r+1}-\sigma_r=\nu
\end{equation}
with
\begin{equation}\label{UR}
U(\unz)=\sum_{i=1}^u (A_i-Z_{r+1+i}), \qquad R(\unz)=\sum_{i=1}^r (Z_i-p_{i}).
\end{equation}
Note that, from \eqref{B} and \eqref{nu}, $\unb \in \cd$ so that
\eqref{hjAB} implies
\begin{equation}\label{hj<=M}
\frac{h_j(n)}{N_{j+r+1}}=\frac{A_1 A_2 \ldots A_u}{f(\unb)} \leq 
\cm=\max_{\unz\in \cd} \frac{A_1 A_2 \ldots A_u}{f(\unz)}\cdot
\end{equation}

If $\unz\in \cd$, from \eqref{Zipi}, \eqref{Zirp1} and \eqref{ZiAi},
it follows that
\[
2=p_1\leq Z_i\leq A_u,\qquad 1\leq i \leq r+1+u
\]
so that $f(\unz)$ does not vanish on $\cd$ and $\cd$ is finite. 
Therefore, the maximum $\cm$ defined by \eqref{opt} is
finite; let $\unc$ be a point in $\cd$ where the maximum $\cm$ is
attained. We shall prove that
\begin{equation}\label{UCRC=0}
U(\unc)=R(\unc)=0.
\end{equation}

For that, first we claim that one of the two numbers $U(\unc), R(\unc)$
vanishes. Indeed, assume that $U(\unc)\neq 0$ and $R(\unc)\neq
0$. 
From \eqref{Zipi}, we have $R(\unc) > 0$; thus there exists
$i_0$, $1\leq i_0 \leq r$, such that
\begin{equation}\label{i0}
C_{i_0} \geq p_{i_0}+1 > p_{i_0}.
\end{equation} 
Similarly, from \eqref{ZiAi}, we have $U(\unc) > 0$, and there exists $i_1$,
$1\leq i_1 \leq u$ such that 
\begin{equation}\label{i1}
A_{i_1} > A_{i_1}-1 \geq C_{r+1+i_1}.
\end{equation} 
Let us define $\unc'\in \N^{r+1+u}$ by
\[
C'_{i_0}=C_{i_0}-1,\;C'_{r+1+i_1}=C_{r+1+i_1}+1,\; 
C'_{i}=C_{i}\; \text{ for } i\neq i_0,r+1+i_1.
\]
To prove that $\unc'\in \cd$, we have to check that $C'_{i_0} \geq
p_{i_0}$ (which follows from \eqref{i0}), that $C'_{i_0} < C_{r+1}$
(which follows from $C'_{i_0} =C_{i_0}-1 < C_{r+1}-1$), that 
$C_{r+1} < C'_{r+1+i_1}$ (since $C_{r+1} < C_{r+1+i_1}$ and 
$C_{r+1+i_1} = C'_{r+1+i_1}-1$), that $C'_{r+1+i_1} \leq A_{i_1}$ 
(which follows from \eqref{i1}) and that 
$U(\unc')-R(\unc')=U(\unc)-R(\unc)$ (which follows from
$U(\unc')=U(\unc)-1$ and $R(\unc')=R(\unc)-1$).
Further,  we have
\begin{eqnarray}\label{fC'}
\frac{f(\unc')}{f(\unc)}&=&\frac{C'_{i_0}C'_{r+1+i_1}}{C_{i_0}C_{r+1+i_1}}
=\frac{(C_{i_0}-1)(C_{r+1+i_1}+1)}{C_{i_0}C_{r+1+i_1}}\notag\\
&=&1-\frac{(C_{r+1+i_1}-C_{i_0}+1)}{C_{i_0}C_{r+1+i_1}} < 1
\end{eqnarray}
since, from the definition of $\cd$ (cf. \eqref{Zirp1} and
\eqref{ZiAi}), $C_{i_0} < C_{r+1} < C_{r+1+i_1}$ holds. But
\eqref{fC'} contradicts the fact that the maximum in \eqref{opt} is
attained in $\unc$.

Let us show now that it is impossible to have simultaneously $U(\unc)
>0$ and $R(\unc)=0$; indeed, let us assume that $U(\unc)\geq 1$ and
$R(\unc)=0$ (which implies $r=0$ or $C_i=p_i$ for
$1\leq i \leq r$). We define $i_1$ as in
\eqref{i1}. Since $\unc\in\cd$, we get from \eqref{URnu}
\[
C_{r+1}=\sigma_{j+r+1}-\sigma_r-\nu+U(\unc)=
(\sigma_{j+r+1}-\sigma_{r+1}-\nu)+p_{r+1}+U(\unc)
\]
which, by \eqref{nu<=n} and $U(\unc) \geq 1$, yield
\begin{equation}\label{Crp12}
C_{r+1} > p_{r+1}+U(\unc) \geq p_{r+1}+1.
\end{equation}
We define $\unc'\in \N^{r+1+u}$ by
\[
C'_{r+1}=C_{r+1}-1,\;C'_{r+1+i_1}=C_{r+1+i_1}+1,\; 
C'_{i}=C_{i}\; \text{ for } i\neq r+1,r+1+i_1.
\]
To prove that $\unc'\in \cd$, we have to check 
that $C'_{i} \geq p_{i}$ for $1\leq i \leq r+1$ 
(which follows from $C'_i=C_i=p_i$ if
$i\leq r$ and from \eqref{Crp12} if $i=r+1$), 
that $C'_i < C'_{r+1}$ for $1\leq i \leq r$ (which follows
from $C'_i=C_i=p_i\leq p_r$ and from
$C'_{r+1}=C_{r+1}-1\geq p_{r+1}$, via\eqref{Crp12}),
that $C'_{r+1} < C'_{r+1+i}$ for $1\leq i \leq u$
(which follows from $C'_{r+1} < C_{r+1}$ and $C'_{r+1+i} \geq   C_{r+1+i}$), 
that $C'_{r+1+i_1} \leq A_{i_1}$ 
(which follows from \eqref{i1}) 
and that $U(\unc')-C'_{r+1}=U(\unc)-C_{r+1})$ (which is easy).
As in \eqref{fC'},  we have  $f(\unc') < f(\unc)$,
contradicting the fact that the maximum in \eqref{opt} is attained in
$\unc$.

To prove \eqref{UCRC=0}, it remains to show that we cannot have
$R(\unc) > 0$ and $U(\unc)=0$. Let us suppose that 
$R(\unc)\geq 1$ and $U(\unc)=0$, 
which implies $u=0$ or, for $1\leq i \leq u$,
\begin{equation}\label{Crp13}
C_{r+1+i}=A_i\geq p_{j+r+2}\geq p_{j+r+1}+2,
\end{equation}
with the help of \eqref{B}.
From \eqref{URnu} and \eqref{nu>=}, this time we get
\begin{eqnarray}\label{Crp14}
C_{r+1} = \sigma_{j+r+1}-\sigma_r -\nu -R(\unc)  &\leq&  \sigma_{j+r+1}-\sigma_r -(\sigma_{j+r}-\sigma_r)
-R(\unc) \notag\\
&=&p_{j+r+1}-R(\unc) \leq p_{j+r+1}-1.
\end{eqnarray}
Here we choose $i_0$ as in \eqref{i0} and set
\[
C'_{i_0}=C_{i_0}-1,\;C'_{r+1}=C_{r+1}+1,\; C'_{i}=C_{i}\; \text{ for
} i\neq i_0,r+1.
\]

To prove that $\unc'\in \cd$, we have to check 
that $C'_{i_0} \geq p_{i_0}$ 
(which follows from \eqref{i0}), 
that $C'_{r+1} \geq p_{r+1}$
(which follows from $C'_{r+1} =C_{r+1}+1$ and  $C_{r+1}\geq p_{r+1}$), 
that, for $1\leq i \leq r$, $C'_i < C'_{r+1}$ 
(which follows from $C'_i \leq C_i$ and $C'_{r+1} > C_{r+1}$),
that, for $1\leq i \leq u$, $C'_{r+1}< C_{r+1+i_1} = A_{i_1}$ 
(which follows from \eqref{Crp13} and \eqref{Crp14}) 
and that $R(\unc')+C'_{r+1}=R(\unc)+C_{r+1}$ (which is easy).
As precedingly in \eqref{fC'}, we observe that $f(\unc') <f(\unc)$, contradicting the fact that the minimum is attained in $\unc$.

In conclusion, we  have proved \eqref{UCRC=0} so that $C_i=p_i$ for
$1\leq i \leq r$ and $C_{r+1+i}=A_i$ for $1\leq i \leq u$. Moreover, 
\eqref{nu<=n} yields $\nu \leq n$, and, from
\eqref{URnu} and \eqref{n'}, we get
\[
C_{r+1}=\sigma_{j+r+1}-\sigma_r-\nu \geq
\sigma_{j+r+1}-\sigma_r-n=p_{j+r+1}-n'.
\]
Therefore, the maximum $\cm$ in \eqref{opt} satisfies
\[
\cm=
\frac{A_1 A_2 \ldots A_u}{p_1 p_2 \ldots p_r C_{r+1} A_1 A_2 \ldots A_u}
\leq \frac{1}{p_1 p_2 \ldots p_r (p_{j+r+1}-n')}
\]
which, via \eqref{hj<=M}, proves \eqref{hj<=}.
\end{proof}

\begin{prop}\label{propminhj}
With the notation of Proposition \ref{propbouhj} , we have
\begin{equation}\label{hj>=}
h_j(n)  \geq   \frac{N_{j+r+1}}{N_r (p_{j+r+1}-n')^\star}=\frac{N_{j+r+1}}{q N_r }
\end{equation} 
where $q=(p_{j+r+1}-n')^\star$ is the smallest prime satisfying $q\geq p_{j+r+1}-n'$.
\end{prop}

\begin{proof}
From \eqref{n'}, we have $p_{r+1} < p_{j+r+1}-n' \leq p_{j+r+1}$ which
implies $p_{r+1} \leq q\leq p_{j+r+1}$ so that  $M=N_{j+r+1}/ (q N_r )$
is an integer with exactly $j$ prime factors. Further, by \eqref{n'},
we have
\[
\ell(M)=\sigma_{j+r+1}-\sigma_r-q\leq
\sigma_{j+r+1}-\sigma_r-(p_{j+r+1}-n')=n
\]
and, by \eqref{hj}, $h_j(n) \geq M$ holds.
\end{proof}

\begin{coro}\label{corohjq}
We keep the notation of Proposition \ref{propbouhj}; if
$q=p_{j+r+1}-n'$ is prime then 
\begin{equation}\label{hj=}
h_j(n) =  h_j(\sigma_{j+r+1}-\sigma_r-q)=\frac{N_{j+r+1}}{q N_r }\cdot
\end{equation} 
\end{coro}

\begin{proof}
Corollary \ref{corohjq} follows from Propositions \ref{propbouhj}
and \ref{propminhj}.
\end{proof}

\section{A parity phenomenon}\label{parparphe}

\begin{prop}\label{propparphe1}
Let $k\geq 2$ be an integer and  $a$ be an even number satisfying 
$4 \leq a < p_{k+1}$ and $h_k$ defined by \eqref{hj}. We have
\begin{equation}\label{hkSk+1-a}
h_k(\sigma_{k+1}-a)=h_k(\sigma_{k+1}-a-1).
\end{equation} 
\end{prop}

\begin{proof}
Since $n\mapsto h_k(n)$ in non-decreasing, we have
\begin{equation}\label{propparphe11}
h_k(\sigma_{k+1}-a) \geq h_k(\sigma_{k+1}-a-1).
\end{equation} 
Let us set $n=\sigma_{k+1}-a$ and note that $n$ satisfies $\sigma_k <
n < \sigma_{k+1}$ so that, from \eqref{k}, $k=k(n)=k(n-1)$.  Let $M$
be a positive squarefree integer such that $\ell(M)\leq n$ and
$\om(M)=k$. Such a number $M$ is even; if not, we would have $\ell(M)
\geq 3+5+\ldots + p_{k+1}=\sigma_{k+1}-2$ in contradiction with
$\ell(M)\leq n=\sigma_{k+1}-a \leq \sigma_{k+1}-4$. Therefore,
$\ell(M)$ is the sum of $2$ and $k-1$ odd numbers, so that
$\ell(M) \equiv \sigma_k \equiv \sigma_{k+1}+1 \equiv \sigma_{k+1}-a-1
\pmod{2}$.  So, $\ell(M)$ cannot be equal to $\sigma_{k+1}-a$ and
$\ell(M)\leq \sigma_{k+1}-a-1$ holds. Thus, from \eqref{hj}, we get
$h_k(\sigma_{k+1}-a) \leq h_k(\sigma_{k+1}-a-1)$, which, with
\eqref{propparphe11}, proves \eqref{hkSk+1-a}.
\end{proof}

\begin{prop}\label{propparphe2}
Let $k$ be an integer, $k\geq 2$, and  $q$ a prime number satisfying 
$3 \leq q \leq p_{k}$. By setting $m=\sigma_{k+1}-q-1$, we have
\begin{equation}\label{hkb}
h_{k-1}(m)=h_{k-1}(m-1)= h_{k-1}(\sigma_{k+1}-q-2)=
\frac{N_{k+1}}{2q} \cdot
\end{equation} 
\end{prop}

\begin{proof}
From the table of Figure \ref{figtabhj}, we have $h_1(6) = h_1(5) = 5$,
$h_2(11) = h_2(10) = 21$ and $h_2(13) = h_2(12) = 35$ so 
that the proposition is true for $k=2$, $q=3$ and for
$k=3$ and $q=3$ or $5$. So, from now on, we assume $k\geq 4$.
Corollary \ref{corohjq} with $j=k-1$, $r=1$ implies 
$h_{k-1}(m-1)=N_{k+1}/(2q)$ and, since $n \mapsto h_{k-1}(n)$
is non-decreasing, it follows that
\begin{equation}\label{hkb>}
h_{k-1}(m) \geq  h_{k-1}(m-1)=
\frac{N_{k+1}}{2q} \cdot
\end{equation} 
Let $M$ be a positive squarefree integer satisfying $\ell(M) \le m$
and $\omega(M) = k-1$. In view of \eqref{hkb>} and \eqref{hj},
to prove that $h_{k-1}(m) = N_{k+1}/(2q)$, it suffices to show that
\begin{equation}\label{hkb>bis}
M \le \frac{N_{k+1}}{2q}\cdot
\end{equation}
If $M$ is odd, $\ell(M)$ is
the sum of $k-1$ odd numbers, which implies
\[
\ell(M)\equiv \sigma_k\equiv \sigma_{k+1}-q = m+1 \pmod{2}.
\]
So, $\ell(M)$ cannot be equal to $m$; since, by \eqref{hj}, $\ell(M) \leq
m$ holds, we should have $\ell(M) \leq m -1$; therefore, from
\eqref{hj}, we get
\begin{equation}\label{h''<}
M \le  h_{k-1}(m-1)= \frac{N_{k+1}}{2q} \cdot
\end{equation} 
If $M$ is even, we have $\om(M/2)=k-2$ and $\ell(M/2)\leq
m-2$, so that
\begin{equation}\label{M2h}
M \le 2 h_{k-2}(m-2).
\end{equation} 
\begin{itemize}
\item
If $q\geq 11$, since we have assumed $k\geq 4$,
i.e. $p_{k+1}\geq 11$, we have
\[
\sigma_k-5 \leq m-2 =\sigma_{k+1}-q-3 
\le \sigma_{k+1}-14
< \sigma_{k+1}-10.
\]
By Proposition \ref{propbouhj} with $j=k-2$, $r=2$, $n=m-2$,
$n'=n-(\sigma_k-\sigma_2)=p_{k+1}-q+2$, we get
\[
h_{k-2}(m-2)\leq \frac{N_{k+1}}{6(q-2)}
\]
which, by \eqref{M2h}, gives
\begin{equation}\label{h'<1}
M  \leq \frac{N_{k+1}}{3(q-2)} = 
\frac{N_{k+1}}{3q} \frac{q}{q-2} \leq \frac{N_{k+1}}{3q} \frac{11}{9} <
\frac{N_{k+1}}{2q} \cdot
\end{equation} 
\item If $q\in \{3,5,7\}$, since $k\geq 4$ and $p_{k+2}\geq
p_6=13$, we have
\[
\sigma_{k+1}-10 \le m-2 = \sigma_{k+1}-q-3 \le \sigma_{k+1}-6
< \sigma_{k+2}-17.
\]
and Proposition \ref{propbouhj} with $j=k-2$, $r=3$, $n=m-2$
and  $n'=n-(\sigma_{k+1}-10)=7-q$ yields
\begin{equation}\label{hkm2}
h_{k-2}(m-2) \leq  \frac{N_{k+2}}{30(p_{k+2}+q-7)} \cdot
\end{equation} 
Using $q\leq 7$, $k\geq 4$ and $p_{k+2}\geq p_6=13$ gives
\begin{eqnarray*}
\frac{N_{k+2}}{30(p_{k+2}+q-7)} &=&\frac{N_{k+1}}{30}
\frac{p_{k+2}}{p_{k+2}+q-7}
\leq \frac{N_{k+1}}{30} \frac{13}{q+6}\\
&=&
\frac{13N_{k+1}}{30 q} \frac{q}{q+6} \leq
\frac{13N_{k+1}}{30 q} \frac{7}{13} < \frac{N_{k+1}}{4q}
\end{eqnarray*}
which, together with \eqref{M2h} and \eqref{hkm2}, proves
\begin{equation}\label{h'<2}
M <   \frac{N_{k+1}}{2q} \cdot
\end{equation} 
Inequalities \eqref{h''<}, \eqref{h'<1} and \eqref{h'<2} prove
that \eqref{hkb>bis} holds, which,
with \eqref{hkb>}, completes the proof of \eqref{hkb}.
\end{itemize}
\end{proof}

\begin{prop}\label{propparphe3}
Let $k$ be a positive integer and  $m=\sigma_{k+1}-1$; we have
\begin{equation}\label{hkb2}
h_{k}(m)=h_{k}(m-1)= h_{k}(\sigma_{k+1}-2)=\frac{N_{k+1}}{2} \cdot
\end{equation} 
\end{prop}

\begin{proof}
It is the same proof than for Proposition \ref{propparphe2}. By
Proposition \ref{propSjSr} with $j=k$ and $r=1$, we have
\[
h_{k}(m) \geq  h_{k}(m-1)=\frac{N_{k+1}}{2} \cdot
\]
Further, let $M$ be a positive integer satisfying
$\ell(M) \le m$ and $\omega(M) = k$.
If $M$ is odd, by the parity phenomenon, we have
$\ell(M)\equiv m-1 \pmod{2}$ so that $\ell(M) \leq m -1$ and
$M \le  h_{k}(m-1)= \frac{N_{k+1}}{2} \cdot$.
If $M$ is even, we have $M\leq 2 h_{k-1}(\sigma_{k+1}-3)$ and,
if $k\geq 2$, i.e. $p_{k+2}\geq 7$,
Proposition \ref{propbouhj} with $j=k-1$, $r=2$, $n=\sigma_{k+1}-3$,
$n'=2$, yields
\[
M\leq 2\frac{N_{k+2}}{6(p_{k+2}-2)} = 
\frac{N_{k+1}}{3} \frac{p_{k+2}}{p_{k+2}-2} \le
\frac{7 N_{k+1}}{15} < \frac{N_{k+1}}{2}\cdot
\]
If $k=1$, it is easy to check that \eqref{hkb2} still holds.
\end{proof}

\section{The increasingness of $h_j(n)$ on $j$}\label{parinchj}

\begin{theorem}\label{thinchj}
Let $n\geq 2$ be an integer and $k=k(n)$ be defined by \eqref{k}; for
$j$ satisfying $1\leq j \leq k$, we have
\begin{equation}\label{hj56}
h_{j-1}(n) \leq \frac 56 h_j(n)
\end{equation}
and \eqref{hj56} is an equality if and only if $j = k(n) \geq 2$ and 
$n=\sigma_{j+1}-4$ or $n=\sigma_{j+1}-5$.
\end{theorem}

\begin{proof}
If $j=1$, it follows from \eqref{h0} and \eqref{h1} that $h_0(n)=1$,
$h_1(n)=\ppn n\geq 2$ and $h_0(n)/h_1(n) \leq 1/2 < 5/6$,
which proves \eqref{hj56}.
So, from now on, we assume $j\geq 2$.

The sequence $(\sigma_{j+r}-\sigma_r)_{r\geq 0}$ is increasing and goes to infinity.
So, we may define $r_j\geq 0$ and $n'_j$ by
\begin{equation}\label{encnj}
\sigma_{j+r_j}-\sigma_{r_j} \leq n < \sigma_{j+r_j+1}-\sigma_{r_j+1}
\end{equation}
and
\begin{equation}\label{n'jdef}
n'_j= n-(\sigma_{j+r_j}  - \sigma_{r_j}).
\end{equation}
We shall consider four cases~: $r_j\leq j-4$,  $r_j \geq j+3$, $j-3
\leq r_j \leq j+2$ and $j\geq 25$,  $j-3
\leq r_j \leq j+2$ and $j\leq 24$.

\subsection*{First case : $r_j\leq j-4$}

From \eqref{pipi1} and our hypothesis $j\geq r_j+4$, we deduce
\begin{equation}\label{prj1}
p_{r_j+1} +p_{r_j+2} \leq p_{2r_j+3} < p_{j+r_j}< p_{j+r_j+1}
\end{equation}
and
\begin{equation}\label{prj2}
p_{r_j+2} +p_{r_j+3} \leq p_{2r_j+5} < p_{j+r_j+2}.
\end{equation}

Let us set 
From \eqref{encnj}, we get
\begin{equation}\label{n'j}
0 \leq n'_j= n-(\sigma_{j+r_j}-\sigma_{r_j}) < p_{j+r_j+1}-p_{r_j+1}
\end{equation}
and applying Proposition \ref{propminhj} yield
\begin{equation}\label{f3}
h_j(n) \geq \frac{N_{j+r_j+1}}{q N_{r_j}}=\frac{N_{j+r_j+1}}{(p_{j+r_j+1}-n'_j)^*\; N_{r_j}}.
\end{equation}
In view of bounding $h_{j-1}(n)$, we have to determine $r_{j-1}$ such
that
\begin{equation}\label{encnj1}
\sigma_{j-1+r_{j-1}}-\sigma_{r_{j-1}} \leq n < \sigma_{j+r_{j-1}}-\sigma_{r_{j-1}+1}.
\end{equation}
We shall distinguish two sub cases.

\subsubsection*{Sub case one, $r_{j-1}=r_j+1$}
Let us asume that 
\begin{equation}\label{f4}
\sigma_{j+r_{j}}-\sigma_{r_{j}} \leq n < \sigma_{j+r_{j}+1}-\sigma_{r_{j}+2}.
\end{equation}
i.e. from \eqref{n'jdef}, 
\begin{equation}\label{f4b}
0\leq n'_j = n-(\sigma_{j+r_j}-\sigma_{r_j}) < p_{j+r_{j}+1}-p_{r_{j}+1}-p_{r_{j}+2}.
\end{equation}
Note that, from \eqref{prj1}, the right hand side of \eqref{f4b} is positive.
Then, we have $r_{j-1}=r_j+1$ since, from \eqref{f4}, 
\begin{eqnarray*}
\sigma_{(j-1)+(r_j+1)}-\sigma_{r_j+1} &\!\! = \!\!&
\sigma_{j+r_j}-\sigma_{r_j+1} < \sigma_{j+r_j}-\sigma_{r_j} \leq n \\
&\!\! < \!\!&  \sigma_{j+r_j+1}-\sigma_{r_{j}+2}   = \sigma_{(j-1)+(r_j+1)+1}-\sigma_{(r_{j}+1)+1}
\end{eqnarray*}
holds. Via \eqref{n'jdef}, this implies that 
\[
n'_{j-1} \stackrel{def}{=\!=}
n-(\sigma_{j-1+r_{j-1}}-\sigma_{r_{j-1}})=n-\sigma_{j+r_{j}}+\sigma_{r_j+1}=n'_j+p_{r_j+1}.
\]
Applying Proposition \ref{propbouhj} and noting that
$j-1+r_{j-1}=j+r_j$  yield
\[
h_{j-1}(n) \leq \frac{N_{j-1+r_{j-1}+1}}{N_{r_{j-1}} (p_{j-1+r_{j-1}+1}-n'_{j-1})}
=\frac{N_{j+r_j+1}}{N_{r_{j}+1} (p_{j+r_{j}+1}-n'_{j}-p_{r_j+1})}
\cdot
\]
By using \eqref{f3}, we get
\begin{equation}\label{hj1/hj}
\frac{h_{j-1}(n)}{h_j(n)}\leq
\frac{(p_{j+r_j+1}-n'_j)^\star}{p_{r_j+1}(p_{j+r_j+1}-n'_j-p_{r_j+1})}\cdot
\end{equation}
From \eqref{n'j}, we have $p_{j+r_j+1}-n'_j > p_{r_j+1} \geq p_1=2$, so
that we may apply Lemma \ref{lem11/8} which, with the help of
\eqref{hj1/hj} and \eqref{f4b}, yields
\[
\frac{h_{j-1}(n)}{h_j(n)}\leq \frac{11}{8 p_{r_j+1}}\left(1+
\frac{p_{r_j+1}}{p_{j+r_j+1}-n'_j-p_{r_j+1}}\right)
< \frac{11}{8}\left(\frac{1}{p_{r_j+1}}+\frac{1}{p_{r_j+2}}\right).
\]
If $r_j \geq 1$,
$\frac{11}{8}\left(\frac{1}{p_{r_j+1}}+\frac{1}{p_{r_j+2}}\right) \leq
\frac{11}{8}\left(\frac{1}{3}+\frac{1}{5}\right) < \frac 56$, which
proves \eqref{hj56}.
\medskip

It remains to consider the case $r_j=0$, which, from
\eqref{f4} and \eqref{k}, implies
$\sigma_j \le n < \sigma_{j+1}$ and  $k(n) = j$.

\eqref{f4b} becomes $0 \le n'_j = n - \sigma_j < p_{j+1}-5$ and, 
by setting $a=p_{j+1}-n'_j$, we get
\begin{equation}\label{a>5}
5 < a=p_{j+1}-n'_j= p_{j+1} + \sigma_j-n = \sigma_{j+1}-n < p_{j+1}
\end{equation}
while \eqref{hj1/hj}  yields
\begin{equation}\label{hjquo}
\frac{h_{j-1}(n)}{h_j(n)}\leq \frac{a^\star}{2(a-2)}\cdot
\end{equation}
By Lemma \ref{lem11/8}, $a^\star \leq \dfrac{11}{8} a$ holds, and, for $a
\geq 12$, $\dfrac{11}{16} \dfrac{a}{a-2} \leq \dfrac{11}{16}
\dfrac{12}{10} < \dfrac 56$, which, via \eqref{hjquo}, proves
\eqref{hj56}.

Since, from \eqref{a>5}, $a > 5$, it remains to study the cases $6\leq
a \leq 11$. If $a=7,9,10,11$, it is easy to check that
$\dfrac{a^\star}{2(a-2)} < \dfrac 56\cdot$
\medskip

If $a=6$ or $a=8$, by Proposition \ref{propparphe1}, \eqref{a>5}
and \eqref{hj>=}  we have
\[
h_j(n)=h_j(\sigma_{j+1}-a)=h_j(\sigma_{j+1}-a-1)=h_j(n-1)
\ge \frac{N_{j+1}}{(a+1)^*}
\]
while, by Proposition \ref{propparphe2}, since $a-1$ is prime, we get
\[
h_{j-1} (n)=h_{j-1}(n-1) = \frac{N_{j+1}}{2(a-1)}
\]
yielding
\[
\frac{h_{j-1}(n)}{h_j(n)}=\frac{h_{j-1}(n-1)}{h_j(n-1)}\leq
\frac{(a+1)^\star}{2(a-1)}=
\begin{cases}
7/10 & \text{ if } a=6\\
11/14 & \text{ if } a=8
\end{cases}
\]
and, in both cases, $\dfrac{h_{j-1}(n)}{h_j(n)} < \dfrac 56$ holds,
which proves \eqref{hj56}.

\subsubsection*{Sub case two, $r_{j-1}=r_j+2$}
Now, we asume that \eqref{encnj} holds but not \eqref{f4}; thus we have
\begin{equation}\label{f5}
\sigma_{j+r_{j}+1}-\sigma_{r_{j}+2} \leq n < \sigma_{j+r_{j}+1}-\sigma_{r_{j}+1}
\end{equation}
and, from \eqref{prj2},
\begin{equation}\label{f4b2}
 p_{j+r_{j}+1}-p_{r_{j}+1}-p_{r_{j}+2} \leq n'_j < p_{j+r_{j}+1}-p_{r_{j}+1}.
\end{equation}
Here, we get $r_{j-1}=r_j+2$, since we have
\[
\sigma_{j-1+r_j+2}-\sigma_{r_j+2}  \leq n < \sigma_{j+r_j+1}-\sigma_{r_{j}+1}   
< \sigma_{j+r_j+2}-\sigma_{r_{j}+3}
\] 
by observing that, from \eqref{prj2},
\[
\sigma_{j+r_j+2}-\sigma_{r_{j}+3} -(\sigma_{j+r_j+1}-\sigma_{r_{j}+1} )=
p_{j+r_j+2}-p_{r_j+2}-p_{r_j+3} > 0
\]
holds. Now, we have
\begin{eqnarray}\label{n'jm1}
0 \leq \;\; n'_{j-1} &=&
n-(\sigma_{j-1+r_{j-1}}-\sigma_{r_{j-1}})=n-\sigma_{j+r_j+1}+\sigma_{r_j+2}\notag\\
&=& n'_j+\sigma_{j+r_j}-\sigma_{r_j}-\sigma_{j+r_j+1}+\sigma_{r_j+2}\qquad\;\;\, 
\text{ from \eqref{n'j}}\notag\\
&=& n'_j-p_{j+r_{j}+1}+p_{r_{j}+1}+p_{r_{j}+2} < p_{r_{j}+2} \qquad
\text{ from \eqref{f4b2}}
\end{eqnarray}
and applying Proposition \ref{propbouhj} gives
\[
h_{j-1}(n) \leq \frac{N_{j-1+r_{j-1}+1}}{N_{r_{j-1}}(p_{j-1+r_{j-1}+1}-n'_{j-1})}
=\frac{N_{j+r_{j}+2}}{N_{r_{j}+2}(p_{j+r_{j}+2}-n'_{j-1})}
\]
while Proposition \ref{propminhj} yields
\[
h_j(n) \geq
\frac{N_{j+r_{j}+1}}{N_{r_{j}}(p_{j+r_{j}+1}-n'_{j})^\star}\cdot
\]
We set $a=p_{j+r_{j}+1}-n'_j$ and
$\De=p_{j+r_{j}+2}-p_{r_{j}+1}-p_{r_{j}+2}$ so that \eqref{n'jm1}
allows to write $p_{j+r_j+2}-n'_{j-1} = \De+a$, and we have
\begin{equation}\label{hjquo2}
\frac{h_{j-1}(n)}{h_j(n)}\leq 
\frac{p_{j+r_j+2}}{p_{r_j+1}p_{r_j+2}} \;\frac{a^\star}{\De+a}\cdot
\end{equation}
\eqref{n'jm1} can be rewritten as
\begin{equation}\label{n'jm1b}
2=p_1\leq p_{r_j+1} < a = p_{j+r_j+1} -n'_j \leq  p_{r_j+1}+p_{r_j+2}.
\end{equation}
Lemma \ref{lem11/8} implies $a^\star \leq 11 a/8$ and, by \eqref{prj1},
$\De > 0$ holds, so that the homographic function $t\mapsto t/(\De+t)$
is increasing. 
From \eqref{n'jm1b}, we thus have
\[
\frac{a}{\De+a} \le \frac{p_{r_j+1}+p_{r_j+2}}{\De+p_{r_j+1}+p_{r_j+2}}
= \frac{p_{r_j+1}+p_{r_j+2}}{p_{j+r_j+2}}\cdot
\]
Therefore, we get
\[
\frac{h_{j-1}(n)}{h_j(n)}\leq 
\frac{11}{8}\frac{p_{r_j+1}+p_{r_j+2}}{p_{r_j+1}p_{r_j+2}}
=\frac{11}{8}\left(\frac{1}{p_{r_j+1}}+\frac{1}{p_{r_j+2}}\right)
\]
which is smaller than $\dfrac 56$ if $r_j \geq 1$.
\smallskip

It remains to consider the case $r_j=0, r_{j-1}=2$.
Formula \eqref{hjquo2} becomes
\begin{equation}\label{hjquo3}
\frac{h_{j-1}(n)}{h_j(n)}\leq \frac{p_{j+2}}{6} \frac{a^\star}{p_{j+2}-5+a}
\end{equation}
while \eqref{n'jm1b} via \eqref{n'jdef} becomes
\begin{equation}\label{apjrj}
2 < a=p_{j+1}-n'_j=\sigma_{j+1}-n\leq 5.
\end{equation}
Since $j \ge 2$ holds, note that \eqref{apjrj} implies
$\sigma_j \le \sigma_{j+1} - 5 \le n < \sigma_{j+1} -2$,
which shows from \eqref{k} that $k(n) = j$.
\begin{itemize}
\item If $a=5$, since $n=\sigma_{j+1}-a=\sigma_{j+1}-5$, by Corollary
\ref{corohjq} with $r=r_j=0$ and $q=5$,  we get $h_{j}(n)=N_{j+1}/5$, while
Proposition \ref{propSjSr} gives
\[
h_{j-1}(\sigma_{j+1} -5)=h_{j-1}(\sigma_{j+1}
-\sigma_2)=\frac{N_{j+1}}{N_2}=\frac{N_{j+1}}{6}\cdot
\]
Therefore, $h_{j-1}(\sigma_{j+1} -5)/h_{j}(\sigma_{j+1} -5)=5/6$.

\item If $a=4$, by Proposition \ref{propparphe1}, we get 
$h_j(\sigma_{j+1}-4)=h_j(\sigma_{j+1}-5)$ and, by Proposition \ref{propparphe2}, 
$h_{j-1}(\sigma_{j+1}-4)=h_{j-1}(\sigma_{j+1}-5)$ so that
\[
\frac{h_{j-1}(\sigma_{j+1} -4)}{h_{j}(\sigma_{j+1} -4)}=
\frac{h_{j-1}(\sigma_{j+1} -5)}{h_{j}(\sigma_{j+1} -5)}=\frac 56\cdot
\]
\item If $a=3$, Formula \eqref{hjquo3} becomes
\[
\frac{h_{j-1}(n)}{h_j(n)}\leq \frac{p_{j+2}}{2(p_{j+2}-2)} \leq
\frac{7}{10} < \frac 56
\]
since $p_{j+2}\geq p_4=7$. 
\end{itemize}
\subsection*{Second case : $r_j\geq j+3$}
From \eqref{encnj}, we deduce $n \geq
\sigma_{j+r_j}-\sigma_{r_j+1}=\sigma_{(j-1)+(r_j+1)}-\sigma_{(r_j+1)}$
and Proposition \ref{propSjSr}, \eqref{lhjn>=}, implies
\begin{equation}\label{lhjm1n}
\ell(h_{j-1}(n)) \geq \sigma_{j+r_j}-\sigma_{r_j+1}.
\end{equation}
Let us now show that
\begin{equation}\label{q=P+}
q \stackrel{def}{=\!=} P^+(h_{j-1}(n))\geq p_{j+r_j}.
\end{equation} 
Indeed, if $q \leq p_{j+r_j-1}$ holds, since $h_{j-1}(n)$ has $j-1$
prime factors, we should have
\begin{eqnarray*}
\ell(h_{j-1}(n)) & \leq & p_{j+r_j-1}+p_{j+r_j-2}+\ldots +p_{r_j+1}\\
& < & p_{j+r_j}+p_{j+r_j-1}+\ldots +p_{r_j+2} =\sigma_{j+r_j}-\sigma_{r_j+1}
\end{eqnarray*}
which would contradict \eqref{lhjm1n}.

Further, among the $j+1$ primes $p_2=3, p_3, \ldots, p_{j+2}$, there
are certainly two primes $p$ and $p'$ not dividing $h_{j-1}(n)$ and
satisfying $3 \leq p < p' \leq p_{j+2}$. By Lemma \ref{lempi+pi},
\eqref{pipi1}, and \eqref{q=P+}, we get
\begin{equation}\label{pp'q}
p+p' \leq p_{j+1}+p_{j+2} \leq p_{2j+3}\leq p_{j+r_j} \leq q=P^+(h_{j-1}(n))
\end{equation}
and, applying Proposition \ref{proppp'hj} proves
$\dfrac{h_{j-1}(n)}{h_{j}(n)} < \dfrac 56\cdot$

\subsection*{Third case : $j-3 \leq r_j \leq j+2$ and $j \geq 25$}
 The proof is the same than for the second case; only, in
 \eqref{pp'q}, instead of \eqref{pipi1}, we use \eqref{pipib} with
 $b=7$, $i=j+2 \geq 27$~:
\[
p+p' \leq p_{j+1}+p_{j+2} < p_{2j-3}\leq p_{j+r_j} \leq
q=P^+(h_{j-1}(n)).
\]

\subsection*{Fourth case : $j-3 \leq r_j \leq j+2$ and $j \leq 24$}
Here, we have $r_j \leq j+2 \leq 26$ and, from \eqref{encnj}, we get
\[
n < \sigma_{j+r_j+1} \leq \sigma_{51}=5350.
\]
So, for $k \leq 50$, $\sigma_k \leq  n < \sigma_{k+1}$ and $1\leq j \leq k$, we
have computed $h_j(n)$ with the algorithm described in Section
\ref{parcomph} and we have checked that, for $j \ge 2$,
\[
\frac{h_{j-1}(n)}{h_{j}(n)}  \leq\frac 56
\]
always holds, with equality if and only if $j=k(n)$ and 
$n=\sigma_{j+1}-4$ or $n=\sigma_{j+1}-5$.
\end{proof}

\begin{coro}\label{coroh=hk}
For all non-negative integer $n\geq 2$, we have
\begin{equation}\label{h61=hk}
h(n)=h_k(n)
\end{equation}
where $k=k(n)$ is defined by \eqref{k}.
\end{coro}

\begin{proof}
From \eqref{h} and \eqref{hj} we have
\[
h(n)=\max_{0\leq j \leq k(n)} h_j(n)
\]
and Theorem \ref{thinchj} yields $h_0(n) < h_1(n) < \ldots < h_{k(n)}(n)$.
\end{proof}

\section{Computation of $\pi_f(x)$}\label{parpifcomp}

\input{Sx.tex}

\section{The algorithm to calculate $h(n)$}\label{paralgoh}

\subsection{The function $G(p_k,m)$}\label{parintrG}

The function $G(p_k,\!m)$ has been introduced and studied in \cite{DNZ}.

\begin{definition}
Let $p_k$ be the $k$-th prime, for some $k\ge 3$ and $m$ an
integer satisfying $0 \le m \le p_{k+1}-3$.
We define
\begin{equation}\label{G}
G(p_k,m)=\max \frac{Q_1Q_2\ldots Q_s}{q_1q_2\ldots q_s}
\end{equation}
where the maximum is taken over the primes 
$Q_1,Q_2,\ldots,Q_s,q_1,q_2,\ldots,q_s$ ($s\ge 0$) satisfying 
\begin{equation}\label{qsQ}
3\le q_s < q_{s-1} < \ldots < q_1 \le p_k < p_{k+1} \le Q_1 < Q_2 < 
\ldots < Q_s
\end{equation}
and
\begin{equation}\label{Qimqi}
\sum_{i=1}^s (Q_i-q_i)\le m.
\end{equation}
\end{definition} 

The additive function $\ell$ (cf. \S\ref{parh}) can easily be extended
to fractions by setting
\[
\ell(M/N)=\ell(M)-\ell(N)
\]
when $M$ and $N$ are coprime or are both squarefree.
Therefore, the inequality \eqref{Qimqi}  implies
\begin{equation}\label{lGm}
\ell(G(p_k,m))\le m.
\end{equation}

\subsubsection{Properties of $G(p_k,m)$}\label{subG}

Obviously, $G(p_k,m)$ is non-decreasing on $m$ and $G(p_k,2m+1)=G(p_k,2m)$.
The maximum in \eqref{G} is unique (from the unicity of the standard
factorization into primes). 
For small $m$'s, we have
\begin{equation}\label{Gpkm1}
0\le m < p_{k+1}-p_k\quad \LR \quad G(p_k,m)=1. 
\end{equation}
From Proposition 8 of \cite{DNZ}, we have
\begin{equation}\label{Gineq}
\frac{p_{k+1}}{(p_{k+1}-m)^\star}\le G(p_k,m)\le \frac{p_{k+1}}{p_{k+1}-m}.
\end{equation}
Note that if $p_{k+1}-m$ is prime, then \eqref{Gineq} yields the
exact value of $G(p_k,m)$.

\subsubsection{Computation of $G(p_k,m)$ }\label{Glarge}

In \cite[\S 8]{DNZ}, two algorithms are given to calculate $G(p_k,m)$. 

The first one is a combinatorial algorithm. In its first step, the
primes allowed to divide the denominator of $G(p_k,m)$ are
determined. From \eqref{qsQ} and \eqref{Qimqi}, they are all the
primes in the range \mbox{$[(p_{k+1}-m),\, p_k]$}, say $P_1 < P_2 < \ldots
< P_K$. Similarly, the primes authorized to divide the numerator are
all the primes 
$P_{K+1} <  P_{K+2} < \ldots < P_R$ in 
\mbox{$[p_{k+1},\, p_k+m]$}.  By setting $\cp'= \{P_1,P_2,\ldots,P_R\}$,
from the definition \eqref{hjP'}, we get
\[
G(p_k,m)=\frac{1}{P_1 P_2 \ldots P_K} h_K(P_1+P_2+\ldots
+P_K+m,\cp')
\]
and $h_j(n,\cp')$ can be computed by induction on $j$ in a way similar
to that exposed in \S \ref{parcomph}. In \cite[\S 8]{DNZ}, one can find the
details and also some tricks to improve the running time of this
combinatorial algorithm which, however, remains rather slow when $m$
is large.

The second algorithm, which is more sophisticated,
is based on the following remark~ :
if 
$G(p_k,m)=\dfrac{Q_1Q_2\ldots Q_s}{q_1q_2\ldots q_s}$ and $m$ is large,
the least prime factor $q_s$ of the denominator is close to $p_{k+1}-m$
while all the other primes $Q_1,\ldots, Q_s,q_1,\ldots,$ $q_{s-1}$ 
are close to $p_k$. 

More precisely,
the following proposition (which is Proposition 10 of \cite{DNZ}) 
says that if 
$p_{k+1}-m + \delta$ is prime for some small $\delta$
and if $G(p_{k+1},\delta)$ is not too small, 
then the computation of $G(p_{k},m)$ is reduced to the computation
of $G(p_{k+1},m')$ for few small values of $m'$, which can be done by
the above combinatorial algorithm.

\begin{prop}\label{Gnew}
We want to compute $G(p_k,m)$ as defined in \eqref{G} 
with $p_k$ odd and $p_{k+1}-p_k\le m \le p_{k+1}-3$. We 
assume that we know some even non-negative integer $\de$ satisfying
\begin{equation}\label{dp}
p_{k+1}-m+\de\quad\text{ is prime, }
\end{equation}
\begin{equation}\label{di}
G(p_{k+1},\de) \ge 1+ \frac{\de}{p_{k+1}}
\end{equation}
and
\begin{equation}\label{dm}
\de < \frac{2m}{9} < \frac{2p_{k+1}}{9}\cdot
\end{equation}
If $\de=0$, we know from \eqref{Gineq} that 
$G(p_k,m)=\dfrac{p_{k+1}}{p_{k+1}-m}\cdot$ If $\de > 0$, we have
\begin{equation}\label{Gq^}
G(p_k,m)=\max_{\substack{q\;\text{ prime}\\p_{k+1}-m\;\le \;q\;\le \;\qh}} \;\;
\frac{p_{k+1}}{q} \;\,G(p_{k+1},m-p_{k+1}+q),
\end{equation}
where $\qh$ is defined by
\begin{equation}\label{qdef}
\qh=\frac{p_{k+1}p_{k+2}(p_{k+1}-m+\de)}{(p_{k+1}+\de)(p_{k+1}-3\de/2)}
\le p_{k+2}-m+\frac{3\de}{2}\cdot
\end{equation}
\end{prop}

How to compute $G(p_k,m)~\!?$\ The combinatorial algorithm should
be tried if $m$ is small, but it is quadratic in $m$ and
has no chance to terminate if $m$ is larger than, say, $10^6$.
We have no guarantee that the conditions of Prop. \ref{Gnew}
are satisfied. However in all our numerical applications,
we have found $\delta < 1000$ in \eqref{dp} (see \cite[\S 9.2]{DNZ}),
so that, by \eqref{Gq^} and \eqref{qdef}, we have
\[
m-p_{k+1}+q \le m - p_{k+1}+\qh \le p_{k+2}-p_{k+1}+\frac{3\delta}{2}
\]
and, in \eqref{Gq^}, $G(p_{k+1,}m-p_{k+1}+q)$ can
be easily calculated by the combinatorial algorithm.

\subsection{Description of the algorithm to compute $h(n)$}\label{algoh}

To compute $h(n)$, the first step is to determine
$p_k$ and $\sigma_k$ defined by \eqref{k}. This step
is explained in \S\, \ref{parSk} and will furnish
also $p_{k+1}$ and $n'=n-p_k$.

\subsubsection{Computation of $p_k$ and $\sigma_k$}\label{parSk}
\begin{enumerate}
\item
Compute $x = \sqrt{\Li^{-1}(n)}$, so that $\Li(x^2)= n$
and $x \sim \sqrt{n \log n}$.

\item
Using Prop. \ref{piprop}, we compute $\pi_{id}(x)$ in time\\
\dsm{\bigo{x^{2/3}/\log^2 x}} = \dsm{\bigo{n^{1/3}/(\log n)^{-5/3}}}.

\item
To get $\sigma_k$, we have to add (if $\pi_{id}(x) < n$) or to
subtract (if $\pi_{id}(x) > n$) to $\pi_{id}(x)$ the primes between
$x$ and $p_k$, calculated by sieving. In practice, this step is very
short. But we are able to estimate it only under Riemann's
hypothesis.  By lemma \ref{approxpi1}, we have
\[
\Li(p_k^2) - \frac{5}{24\pi}p_k^{3/2}\log p_k < \sigma_k
\le n < \sigma_{k+1} 
< \Li(p_{k+1}^2)  + \frac{5}{24\pi}p_{k+1}^{3/2}\log p_{k+1}
\] 
which implies
\[
n = \Li(p_k^2) + \bigo{p_k^{3/2}\log p_k}
\sim \Li(p_k^2) \sim \frac{p_k^2}{2\log p_k}\cdot
\]
Therefore, we get $\log n \sim 2\log p_k$,
$p_k \sim \sqrt{n \log n}$ and
\begin{equation}\label{ZUT}
\Li(p_k^2) =  n + \bigo{n^{3/4}(\log n)^{7/4}}.
\end{equation}
Further, since $x \sim p_k \sim\sqrt{n \log n}$, we have
\[
\abs{n - \Li(p_k^2)} = \abs{\Li(x^2) - \Li(p_k^2)}
\sim \frac{\abs{x^2-p_k^2}}{2\log x} 
\sim 2\abs{x-p_k}\sqrt{\frac{n}{\log n}},
\]
so that, from \eqref{ZUT},
\dsm{
\abs{\pi(x)-\pi(p_k)} \le \abs{x-p_k} = \bigo{n^{1/4}(\log n)^{9/4}}.
}
\end{enumerate}
\suppress{
By the defintion \eqref{k}, $n$ is near from
$\pi_{id}(p_k)$. Therefore, by the equation \eqref{equapprox} of lemma
\ref{approxpi1}, $p_k$ is near from $\sqrt{\Li^{-1}(n)} \sim \sqrt{n
  \log n}$.  
The simplest way to compute $\sigma_k$ is~: Add all the successive
primes until the sum exceeds $n$. The generation of all the primes up
to $p_k$ by Eratosthene's sieve is of cost $\bigo{\sqrt{n\log n}}$,
which is too expansieve for large values of $n$.
\medskip

Thus, to  compute $p_k$ and $\sigma_k$, we
\begin{enumerate}
\item
Compute $x_k=\sqrt{\Li^{-1}(n)}$, then
$\Li(x_k^2) = n$, $x_k \sim \sqrt{n\log n}$ and by \eqref{equapprox},
\[
\abs{\pi_{id}(x_k) - n} \le \frac{5}{24\pi} x_k^{3/2} \log x_k
\sim \frac{5}{48\pi} n^{3/4} (\log n)^{7/4}. 
\]
\item
Using proposition \ref{piprop} we compute $\pi_{id}(x_k)$ in time
$\bigo{x^{2/3}/\log^2 x_k} = \bigo{n^{1/3} (\log n)^{-5/3}}$.
\item
The values of primes about $x_k$ are nearly $n^{1/2}(\log n)^{1/2}$.
Thus the number of primes to add or subtract to $pi_{id}(x_k)$
for finding the value $\sigma_k$ is about
$\bigo{n^{1/4}}{(\log n)^{5/4}}$.  
According to $\pi_{id} x_k < n$ or $\pi_{id}(x_k) > n$ we sieve a small
interval begining by $x_k$ or ending by $x_k$ to find the exact
values of $p_k$ and $\sigma_k$.
\end{enumerate}
}

\subsubsection{Computation of $h(n)/N_k$}

By Corollary \ref{coroh=hk}, we have $h(n)=h_k(n)$. Let us
set $n' = n-\sigma_k$. If $n'=p_{k+1}-1$
or $n'=p_{k+1}-2$, Proposition \ref{propparphe3} yields
$h(n)=N_{k+1}/2$. So, we may suppose  $n'\leq \sigma_{k+1}-3$. 
From the definition \eqref{G} of function $G$, we have
\begin{equation}\label{h=NkG}
h(n)=h_k(n)=N_k G(p_k,n')
\end{equation}
and we compute $G(p_k,n')$ as explained in \S \ref{Glarge}.  In
practice, the computation of $G(p_k,n)$ is very fast.  However, as
explained in \S \ref{Glarge}, we have no estimation of the running
time.  

Below, are listed some values of $\dfrac{h(n)}{N_k}=G(p_k,n')$
together with $p_k$, $n'=n-\sigma_k$, $e=e(n)$ the largest integer
such that $h(n-e)=h(n)$) and, if the algorithm of Proposition
\ref{Gnew} is used, $\delta$ and $Q$, the number of primes used in the
sum \eqref{Gq^}.
\[
n=10^{12},\  p_k=5477081,\  n'=4935150,\ e=0,\ \de= 18,Q=1,
\]
\[
G(p_k,n')=  \frac{29998525822277}{2968309525031} =
\frac{5477089\times 5477093}{5477081\times 541951}\cdot
\]  
\[
n=10^{35},\  p_k=\ 2898434150644708999,\ n'=\ 1886081812111845520,\ e
= 16
\ 
\]
\[
\de=134,\ Q= 5,
\quad G(p_k,n')= \frac{2898434150644709023}{1012352338532863519}\cdot
\]

The values  of $h(10^{a})$ for $a \le 35$ and of
$h = 2^b$ for $b \le 116$ can be found on the 
authors's web sites \cite{DEL.WEB, NIC.WEB}, together
with the Maple or Sage programs computing
$h(n)/N_k$.

\section{An open question}\label{paropq}

Given $n$ and $j < k(n)$, how to compute $h_j(n)$? We have not
succeeded in solving this problem when $n$ is too large to use the
naive algorithm described in \ref{parcomph}. The case $j=2$ is already
not that simple.

A first step is certainly to calculate $r=r(n,j)$ defined by
\eqref{encn}, which can be done by the method of \S\, \ref{parSk}. If we
are lucky enough that $q=p_{j+r+1}-n'=p_{j+r+1}-(n-\sigma_{j+r}+\sigma_r)$ is
prime, then the value  $h_j(n)= \frac{N_{j+r+1}}{q N_r }$ is given by \eqref{hj=}.

In the general case, by setting $n'=n-(\sigma_{j+r}-\sigma_r)$, one may think
that
\dsm{h'_j(n)=\frac{N_{j+r}}{N_r }G(p_{j+r},n')}
has a good chance to be the value of $h_j(n)$. However, there are
exceptions.


\input{figure.tex}

\def\refname{References}
\bibliography{Bibhden}{}
\bibliographystyle{plain}
\bigskip

\hspace{-2cm}
\begin{minipage}{13cm}{
Marc Del\'eglise, \url{deleglis@math.univ-lyon1.fr}.\\
Jean-Louis Nicolas, \url{jlnicola@in2p3.fr}.\\
Université de Lyon, CNRS,\\
Institut Camille Jordan, Math\'ematiques,
Université Claude Bernard (Lyon 1),\\
21 Avenue Claude Bernard,
F-69622 Villeurbanne c\'edex, France.
}
\end{minipage}

\end{document}

%% file: Sx.tex
Let $f$ be an arithmetic function, i.e a function defined on
positive integers. The simplest way to compute $\pi_f(x)$
defined in \eqref{defsfx} is to generate the primes up to
$x$ by Eratosthenes's sieve, which is too expansive for
large values of $x$. 

\begin{definition}
An arithmetic function $f$ is said to be completely multiplicative
if $f(ab) = f(a) f(b)$ for all $a$ and $b$.
If $f \ne 0$, this implies $f(1) = 1$.
\end{definition}

Following ideas of the german astronomer Meissel, Lagarias, Miller and
Odlyzko gave in \cite{LMO} an algorithm that computes $\pi(x)$ with a
cost \dsm{\bigo{\frac{x^{2/3}}{\log x}}}. In this work they also
remark that their algorithm allows to compute $\pi_f(x)$ for every
completely multiplicative arithmetic function $f$.

This method has been improved in \cite{DRPI} to compute $\pi(x)$ with
a cost \dsm{\bigo{\frac{x^{2/3}}{\log^2 x}}}, provided that all the
arithmetic operations on integers are of constant cost $\bigo 1$, not
depending on the size of the operands.
We show here that this improved algorithm may be used to compute
\dsm{\pi_f(x)} whith a cost which is still
\dsm{\bigo{\frac{x^{2/3}}{\log^2 x}}}, for a large subset of the set
of completely multiplicative arithmetic functions. More precisely we
have the proposition~:

\begin{prop}\label{piprop}
Let $f$ be a completely multiplicative arithmetic function with
integer values.  Let $F$ be the summatory function of $f$,
\begin{equation}\label{defFs}
F(x) = \sum_{n \le x} f(n).
\end{equation}
We suppose that all the ordinary arithmetic operations
about integers are of constant cost $\bigo{1}$, and that
\begin{enumerate}
\item
Each value $f(n)$ may be computed in time $\bigo{1}$,
not depending of the size of $n$.
\item
There is an algorithm computing
\begin{equation}\label{defS0}
S_0(y,x) = \sum_{1\le n \le y} \mu(n)f(n) F\left(\dfrac{x}{n}\right)
\end{equation}
in time \dsm{\bigo{x^{2/3}/\log^2 x}.}
\end{enumerate}
Then, there is an algorithm computing
\dsm{
\pi_f(x) = \sum_{p \le x \atop p \text{ prime}} f(p) 
}
in time \dsm{\bigo{\frac{x^{2/3}}{\log^2 x}}}.

When $F(u)$ can be computed in $\bigo{1}$ time, the second hypothesis is satisfied.
\end{prop}

\medskip\noindent
{\bf Remarks~:}
\begin{enumerate}
\item
The second hypothesis may seem strange.
Let us give a few words of explanation. 

\begin{itemize}
\item
Our computation of $\pi_f(x)$ begins by choosing $y =
\bigo{x^{1/3+\veps}}$.  Then we compute $S_0 = S_0(y,x)$ (this is the
  \emph{contribution of ordinary leaves} defined in lemma 5.2,
  equation (9) in \cite{DRPI} and in lemma 7.2, equation (7.14) in
  this article).
Function $F$ does not appear elsewhere in the algorithm.  $S_0$ being
computed, the total cost of the other computations is
\dsm{\bigo{x^{2/3}/\log^2 x}}.  Condition (2) ensures that our
algorithm computes $\pi_f(x)$ in time \dsm{\bigo{x^{2/3}/\log^2 x}}.
\item
In many cases, $F(u)$ can be computed in time $\bigo{1}$, 
then the sum \eqref{defS0} can be computed in time $\bigo{y}$,
by precomputing the Möbius function, so that the second hypothesis
is satisfied.
\end{itemize}
\item
In Proposition \ref{piprop} we restrict ourseves to the case of
integer valued functions. The case of real valued functions is more
delicate because of truncation errors.  In \cite{BACH2}, Bach and al. 
have elaborated an algorithm to compute $\pi_f(x)$ where $f(n)=1/n$,
and\\
\dsm{x = 1\,801\,241\,484\,456\,448\,000 = 1.8\ldots \times 10^{18}}.
\end{enumerate}
\medskip

\subsection*{Algorithm for $\pi_f(x)$}

We will describe very briefly our algorithm to compute $\pi_f(x) $,
using notations and formulas which, when replacing $f$ by $1$, reduce
to the correponding ones contained in \cite{DRPI}.

For $b \in \N$, let us define $\Phi(x,b)$ as the sum of the $f(n)$, for
the $n's \in \intf{1}{x}$ that subsist after sieving this interval by
all primes $p_1,p_2, \dots, p_b$,

\begin{equation}\label{defPhi}
\Phi(x,0) = \sum_{1 \le n \le x} f(n)  = F(x)
\qtx{and, for $b \ge 1$, } 
\Phi(x,b) = \sum_{1 \le n\le x \atop \Pm(n) > p_b} f(n)
\end{equation}
For $k \ge 1$ and $b \ge 1$, let us set
\begin{equation}\label{defPk}
P_k(x,b) = 
\sum_{1 \le n \le x \atop{\Omega(n)=k,\ \Pm(n) > p_b }} f(n).
\end{equation}
so that, from \eqref{defPhi} and \eqref{defPk} we get, for $x \ge 1$,
\begin{equation}\label{phidec}
\Phi(x,b) = 1 + P_1(x,b) + P_2(x,b) + \cdots
\end{equation}

From now on, we choose $y \in \R$ 
\begin{equation}\label{defy}
x^{1/3} \le y \le \sqrt{x}
\qtx{ and set }
a = \pi(y).
\end{equation}
We will precise later the best choice for $y$, which is closed
to $x^{1/3}$. 

Since $y \ge x^{1/3}$ equation \eqref{defPk} yields $P_k(x,a) = 0$
for $k \ge 3$ and \eqref{phidec} becomes
\begin{equation*}
\Phi(x,a) = 1 + P_1(x,a) + P_2(x,a).
\end{equation*}
Since 
\dsm{P_1(x,a) = \sum_{p_a < p \le x} f(p)
= \sum_{y < p \le x} f(p) = \pi_f(x)-\pi_f(y)}, 
\begin{equation}\label{sfx}
\pi_f(x) = \Phi(x,a) + \pi_f(y) - 1 - P_2(x,a).
\end{equation}
Replacing $f$ by $1$ (and $\pi_f$ by $\pi$), formula \eqref{sfx} is
formula $(4)$ in \cite{DRPI}.

\subsection{Initialization of the computation:~the 2 basis tables}\label{init}

After fixing $y$, by using Eratosthenes's sieve, we precompute
the table of primes up to $y$, and the table of the values
$\pi_f(u)$ for $1 \le u \le y$.  The cost of these initializations is
$\bigo{y \log_2 y}$.

\subsection{Computation of $P_2(x,a)$}
Definition \eqref{defPk} and the complete multiplicativity
of $f$ give
\[
P_2(x,a) = \sum_{y < p \le q \le x \atop pq \le x} f(pq) 
= \sum_{y < p \le q \le x \atop pq \le x} f(p)f(q)
\]
where $p$ and $q$ are primes.
The $p's$ figuring in this sum satisfy
$p \le \dfrac{x}{q} \le \dfrac{x}{y}$ and we get
\[
P_2(x,a) = \sum_{y < p \le x/y} f(p) \sum_{p \le q \le x/p} f(q).
\]
We remark that, for $p > \sqrt{x}$, the sum on $q$ vanishes.
Since, by \eqref{defy}, $\sqrt{x} \le \dfrac{x}{y}$, we have
\[
P_2(x,a) = \sum_{y < p \le \sqrt{x}} f(p) \sum_{p \le q \le x/p} f(q)
= \sum_{y < p \le \sqrt{x}} f(p) 
\left(\pi_f\left(\frac{x}{p}\right)-\pi_f(p-1)\right)
\]
or
\begin{equation}\label{fP2}
P_2(x,a) = \sum_{y < p \le \sqrt{x}} f(p)\pi_f\left(\frac{x}{p}\right)
- \sum_{y < p \le \sqrt x} f(p) \pi_f(p-1).
\end{equation}

In the above formula, the values of $p$ are bounded above by
$\sqrt{x}$ which is larger than $y$. Thus we cannot find these primes
$p$, nor the values $\pi_f(p-1)$ in the precomputed tables (cf. \S
\ref{init}), and we generate them using a sieve of
$\intf{1}{\sqrt x}$, which we call the \emph{auxilliary sieve}.  The
values of $x/p$ lie in the interval $\intf{1}{x/y}$.  So we will get
the values $\pi_f(x/p)$ by an other sieve, the \emph{main sieve}.  Let
us note that the respective sizes of the sieve intervals, $\sqrt x$
and $x/y$, are too large to allow a sieve in one pass. Thus the two
sieves will be done by blocks of size $y$ that must be synchronized.
\medskip

\begin{itemize}
\item{\bf Initialization: Computation of $\varpi$, the largest prime
  $\le \sqrt x$ and of $\pi_f(\varpi)$.}  By Eratosthenes'sieve we
  compute the largest prime $\varpi$ not exceeding $\sqrt x$ and calculate
  $\pi_f(\varpi)$. 
  The auxilliary sieve is then initialised by putting in
  the sieve-table the primes $p$ of the block $[\sqrt x -y+1, \sqrt x]$.
  The main sieve is initialized by sieving the first block
  \dsm{[A,B] = \intf{\sqrt x}{\sqrt x + y -1}}.
  The cost of this phase is \dsm{\bigo{x^{1/2}\log_2 x}}.

\item{\bf Computing $P_2(x,a)$.}  
  We use formula \eqref{fP2}, getting in decreasing order the primes
  $p \in\, ]y,\sqrt x]$ and the $ f(p) \pi_f(p-1)$ from the auxillary sieve,
  and getting the values $\pi_f(x/p)$ from the main sieve whose
  successive blocks will cover in ascending order the
  interval \dsm{\intfg{\sqrt x}{x/y}}.

  We initialize a variable $p$ with the value $\varpi$, 
  a variable $T$ with the value $\pi_f(\varpi)$ and
  a variable $P_2$ with the value $0$.
  Then, while
  $p > y$, we repeat~:
\begin{itemize}
 \item substract $f(p)$ from $T$. Thus the new value of $T$ is
   $\pi_f(p-1)$.  \item If $x/p > B$, while $x/p > B$ we replace the
   block $[A,B]$ by the next block $[A+y, B+y]$ and we sieve
   it. When $x/p \in [A,B]$ we get $\pi_f(x/p)$ in the main sieve
   table and we add $f(p)\pi_f(x/p)- f(p) T$ to $P_2$.  
\item Using the auxilliary sieve, replace $p$ by its predecessor.
\end{itemize}
\end{itemize}
The final value of the variable $P_2$ is $P_2(x,a)$.  The first step
is negligible in cost, compared to the second.  Thus the computation
of $P_2(x,a)$ is of total cost \dsm{\bigo{\frac{x}{y}\log_2 y}}.

\subsection{Computation of $\Phi(x,a)$}
The following lemma is proved as lemma 5.1 in \cite{DRPI}. 

\begin{lem}
For every $u \ge 0,$ and for $b \ge 1$,
\begin{equation}\label{recfi1}
\Phi(u,0) = F(u)
\end{equation}
\begin{equation}\label{recfi2}
\Phi(u,b) = \Phi(u,b-1) - f(p_b)\Phi\left(\frac{u}{p_b},b-1 \right)
\end{equation}
\end{lem}

This relation gives an obvious method for computing
\mbox{$\Phi(x,a)$}. Starting from the tree with the only node
\mbox{$\Phi(x,a)$}, and applying repeatedly \eqref{recfi2}
we get a tree whose all nodes, except the root node,
are labelled by a formula of the form
\begin{equation}\label{gennode}
\mu(n)f(n) \Phi\left(\frac{x}{n},b\right)
\end{equation}
where $b \le a-1$
and $n=1$ or $n$ is a squarefree integer with prime factors 
$q\in \set{p_{b+1}, \dots, p_a}$.

If we repeat this expansions until all the leaves of the resulting tree
are labelled by formulas \mbox{$\mu(n)f(n) \Phi\left(\frac{x}{n},0\right)$},
using \eqref{recfi1}
we get the formula~:
\begin{equation}\label{fullrec}
\Phi(x,a) = 
\sumd{1 \le n \le x}{P^+(n) \le y} 
\mu(n)f(n) \Phi\left(\frac{x}{n},0\right) =
\sumd{1 \le n \le x}{P^+(n) \le y} 
\mu(n)f(n) F\left(\dfrac{x}{n}\right)
\end{equation}
which, when $f=1$ is formula
\dsm{\Phi(x,a) = \sumd{1 \le n \le x}{P^+(n) \le y}\
  \mu(n) \intpart{\frac{x}{n}}}
\ (cf. \cite[p. 237]{DRPI}).
\medskip

The number of terms in \eqref{fullrec} is much too large.  
In order to get a sum with fewer terms 
we replace the trivial rule
\medskip 

\centerline{{\bf Rule 1~:} \emph{ Expand  \eqref{gennode} using
  \eqref{recfi2} if $b > 0$},}

\medskip\noindent
which leads to \eqref{fullrec} by the new rule
\medskip

\centerline{{\bf Rule 2~:} \emph{ Expand node\eqref{gennode}
only if  $b > 0$ and $n \le y$.}}
\bigskip

\noindent
Expanding the computation tree whith rule 2 instead of rule 1
we get 

\begin{lem}
We have
\begin{equation}\label{}
\Phi(x,b) = S_0 + S,
\end{equation}
where $S_0$ is the contribution of \emph{ordinary leaves}
\begin{equation}\label{}
S_0 = S_0(y,x)
= \sum_{1 \le n \le y} \mu(n)f(n)\Phi\left(\frac{x}{n},0\right) 
= \sum_{1 \le n \le y} \mu(n)f(n) F\left(\dfrac{x}{n}\right)
\end{equation}
and $S$, the contribution of \emph{special leaves}, is
\begin{equation}\label{specialsum}
S = \sum_{\frac{n}{\Pm(n)} \le y < n}
\mu(n)f(n) \Phi\left(\frac{x}{n},\pi(\Pm(n))-1\right)\cdot
\end{equation}
\end{lem}

This lemma corresponds to lemma (5.2) in \cite{DRPI}.

\subsubsection{Computation of $S_0$}
In the general case, the computation of $S_0$ is done with a cost
$\bigo{x^{2/3}/\log^2 x}$ thanks to the condition
2 in proposition \ref{piprop}. 

In the case we will consider later in this work,
the computation of $\pi_{id}(x)$,
$f(n) = n$, thus $F(u) = \dfrac{[u][u+1]}{2}$ is computed
in $\bigo{1}$ time and the computation of $S_0(x,y)$ 
is of cost \dsm{\bigo{y} = \smallo{x^{2/3}/\log^2 x}}.

\subsubsection{Computation of $S$}

In the sum \eqref{specialsum}, 
let us set $n = mp$ with $p = \Pm(n)$.
Grouping together all the $n'$s according to the value of $p$,
we get
\begin{equation}\label{specialbis}
S =
-\sum_{p \le y} f(p)\sum_{\Pm(m)> p \atop m \le y < mp}
\mu(m)f(m) \Phi\left(\frac{x}{mp},\pi(p)-1\right)\cdot
\end{equation}

The computation of $S$ from \eqref{specialbis} is the complicate part
of the algorithm.
In the following paragraph we show that it is relatively
simple to get a cost $\bigo{x^{2/3+\veps}}$.

\subsubsection{How to compute $S$ in \dsm{\bigo{x^{2/3+\veps}}}}\label{pifcomp}

In this section, we explain a first method to get $\pi_f(x)$, 
rather simple to implement, and whose running time 
is $\bigo{x^{2/3+\veps}}$.
We take $y=x^{1/3}$.
Since $mp > y$ all the values $u = x/mp$ appearing in
\eqref{specialbis} are less than $x^{2/3}$. We sieve the
interval \dsm{\intfg{1}{\frac{x}{y}}} successively by all primes $p
\le y$. After the sieve by $p$, from
the definition \eqref{defPhi} of $\Phi$, for all the $m$'s such that
$m
\le y < mp$, we get in the sieve table the value
\dsm{\Phi\left(\frac{x}{mp},\pi(p)-1\right)},
and we add to $S$ the value \dsm{f(p)\mu(m)f(m)
  \Phi\left(\frac{x}{mp},\pi(p)-1\right)\cdot}

But, if we proceed in the naive way, after sieving by each $p$, we
will update the sieve table, putting in the case of index $u$ the
sum of $f(n)$ for the $n$'s,  $n \le u$ that are still in the table. This is
excluded because, for each $p$ this would cost $\bigo{x/y}$
operations, and the total cost of these updatings would be $\gg
\pi(y)(x/y) = x/\log x$.  As explained in \cite{LMO} (the 7 last-lines
p. 545 and the first half of p. 546) we use an auxiliary data
structure such that, for a price of $\bigo{\log x}$ time in place of
$\bigo{1}$ for each access, we don't need to update the sieve table
after each sieve.  To be a little more precise let us say that this
structure is a labelled binary tree. There is a leave for each index
$i$ of the table sieve, this leave is labelled by the value $f(i)$,
and each interior node is labelled by the sum of labels of its two
sons.  Proceeding in this way the cost of the sieve
is \dsm{\bigo{\frac{x}{y}\log x\log_2 x}}, while the cost of
retrieving the
values \dsm{f(p)\mu(m)f(m) \Phi\left(\frac{x}{mp},\pi(p)-1\right)}
is \dsm{\bigo{\pi(y)y \log x}}. Both costs
are \dsm{\bigo{x^{2/3+\veps}}} with our choice $y=x^{1/3}$.

\subsubsection{Faster computation of $S$}
In this section, we explain how to carry out the computation
of $\pi_{f}(x)$ in \dsm{\bigo{\frac{x^{2/3}}{{\log^2x}}}}.
We take $y = x^{1/3}(\log x)^3\log_2 x$.
To speed up the computation of $S$ we partition
\eqref{specialbis}
in 3 subsums $ S = S_1 + S_2 + S_3,$

\begin{eqnarray*}
S_1 &=&
-\sum_{x^{\frac{1}{3}} < p \le y} f(p) \sum_{\Pm(m)> p \atop m \le y < mp}
\mu(m)f(m) \Phi\left(\frac{x}{mp},\pi(p)-1\right)\\
S_2 &=&
-\sum_{x^{\frac{1}{4}} < p \le x^{\frac{1}{3}}}f(p) \sum_{\Pm(m)> p \atop m \le y < mp}
\mu(m)f(m) \Phi\left(\frac{x}{mp},\pi(p)-1\right)\\
\end{eqnarray*}
\begin{eqnarray*}
S_3 &=&
-\sum_{p \le x^{\frac{1}{4}}} 
f(p) \sum_{\Pm(m)> p \atop m \le y < mp} \mu(m)f(m) \Phi\left(\frac{x}{mp},\pi(p)-1\right)\\
\end{eqnarray*}

We will show that $S_1$ is quickly computed in $\bigo{y}$ time.  $S_3$
will be computed by sieve, as explained in \S\ref{pifcomp}, but
faster because the number of values for $p$ is reduced from $\pi(y)$
to $\pi(x^{1/4})$.  The main part of the computation will be the
computation of $S_2$.

As in \cite{DRPI}, we first observe that the $m'$s involved in $S_1$
and $S_2$ are all prime and therefore~:

\begin{eqnarray}
\label{S1}
S_1 &=& 
\sum_{x^{\frac{1}{3}} < p \le y}\ f(p) \sum_{p < q \le
  y}f(q)\Phi\left(\frac{x}{pq},\pi(p)-1\right)\\
\label{S2}
S_2 &=& 
\sum_{x^{\frac{1}{4}} < p \le x^{\frac{1}{3}}}f(p) \ \sum_{p < q \le y}f(q)\Phi\left(\frac{x}{pq},\pi(p)-1\right)
\end{eqnarray}

\paragraph{Computing $S_1$}
As in \cite{DRPI} we remark that, in \eqref{S1}, we have
\dsm{\frac{x}{pq} < x^{1/3} < p}. Thus, all the values
\dsm{\Phi\left(\frac{x}{pq},\pi(p)-1\right)} are equal to $1$.
Therefore
\begin{equation*}
S_1 = \sum_{x^{\frac{1}{3}} < p \le y} f(p) \sum_{p < q \le y}f(q)
= \sum_{x^{\frac{1}{3}} < p \le y} f(p) \left(\pi_f(y)-\pi_f(p)\right).
\end{equation*}
This value is computed in $\bigo{y}$ additions, using
the precomputed table of the values $\pi_f(u)$ for $1 \le u \le y$.

\paragraph{Computing $S_3$}

For each $p \le x^{1/4}$ we precompute the list of all
the squarefree $m \le y$ whose least factor is $p$.

We sieve the interval $\intf{1}{\frac{x}{y}}$ successively by all the
primes up to $x^{1/4}$. As soon as we have sieved by $p$, using the
precomputded lists of squarefree whose least prime factor is a prime
$q >p$ we sum the
\[
f(p) \sum_{\Pm(m)> p \atop m \le y < mp} \mu(m)f(m)
\Phi\left(\frac{x}{mp},\pi(p)-1\right)
\]
for all squarefree $m \in \intfg{y/p}{y}$
such that $\Pm(m) > p$. This computation is done by blocks,
using the auxiliary structure, as explained at the end of 
\S~\ref{pifcomp}. 

Thus the cost of sieving is \dsm{\bigo{\frac{x}{y}\log x \log_2 x}}.
The number of values of $p$ is $\pi\big(x^{1/4}\big)$ and the number
of values of $m$ is less than $y$, 
thus the cost of retrieving the values
\dsm{f(p)\mu(m)f(m)
  \Phi\left(\frac{x}{mp},\pi(p)-1\right)}
is \dsm{\bigo{\pi\big(x^{1/4}\big) \times \log x \times y}}
Thus computing $S_3$ is of cost 
\dsm{\bigo{\frac{x}{y}\log x \log_2 x + y x^{1/4}}}

\paragraph{Computing $S_2$}
We split the sum \eqref{S2} in two parts depending on 
$q > x/p^2$ or $q \le x/p^2$. It gives
\[
S_2 = U + V
\]
with
\[
U = \sum_{x^{\frac{1}{4}} < p \le x^{\frac{1}{3}}} f(p)\
\sum_{p < q \le y \atop q > x/p^2}\ f(q)\Phi\left(\frac{x}{pq},\pi(p-1\right)
\]
and
\[
V = \sum_{x^{\frac{1}{4}} < p \le x^{\frac{1}{3}}} f(p)\ \sum_{p < q \le y \atop q \le
  x/p^2}\
\ f(q)\Phi\left(\frac{x}{pq},\pi(p-1\right)
\]

\paragraph{Computing  U} 
With $y < \sqrt x$ (cf. \eqref{defy}),
the condition $q > x/p^2$ implies
$p^2 > x/q \ge x/y \ge x^{1/2}$. Thus,
\[
U = \sum_{\sqrt{x/y} < p \le x^{1/3}}f(p)\ \sum_{p < q \le y \atop q >
  x/p^2}\
\ f(q)\Phi\left(\frac{x}{pq},\pi(p-1\right)
\]
From $x/p^2 < q$ we deduce $x/pq < p$ and $\Phi(x/pq,\pi(p)-1)=1$, and
we have $x/p^2 \ge p$ so that
\[
U = \sum_{\sqrt{x/y} < p \le x^{1/3}}f(p)\ \sum_{p < q \le y \atop q >
  x/p^2} \ f(q) = 
\sum_{\sqrt{x/y} < p \le x^{1/3}}f(p)\
\left(\pi_f(y)-\pi_f\left(\frac{x}{p^2} \right)\right)\cdot
\]
Since $x/p^2 < q \le y$ the sum $U$ is calculated in $\bigo{y}$ 
operations with the table of values of $\pi_f(u)$.

\paragraph{Computing $V$}

For each term involved in $V$ we have
\dsm{
p \le \frac{x}{pq} < x^{1/2} < p^2.
}
Hence, by \eqref{defPhi}, \dsm{\Phi(x/(pq),\pi(p)-1)} is the sum
of $f(n)$ for $n$ satisfying $n \le x/pq$ and $\Pm(n) \ge p$. These
$n$'s are $n=1$ and all the primes $n$ satisfying
$p-1 < n \le x/(pq)$. Thus
\[
\Phi\left(\frac{x}{pq},\pi(p)-1 \right)
= 1 + \pi_f\left(\frac{x}{pq} \right) -\pi_f(p-1)
\]
And we write
\[
V = V_1 + V_2
\]
with
\begin{eqnarray*}
V_1 &=& \sum_{x^{1/4} \le p < x^{1/3}} f(p) \sum_{p < q \le \min(\frac{x}{p^2},y)}
f(q)(1- \pi_f(p-1))\\
V_2 &=& \sum_{x^{1/4} \le p < x^{1/3}} f(p)\
\sum_{p < q \le\min(\frac{x}{p^2},y)}
f(q) \pi_f\left(\frac{x}{pq}\right)
\end{eqnarray*}
Computing $V_1$ can be achieved in $\bigo{y}$ time
once we have tabulated $\pi_f(u)$ for $u \le y$.

\paragraph{Computing $V_2$.}

We first split $V_2$ in two parts in order to
simplify the condition $q \le \min(x/p^2,y)$~:
\[
V_2 =
\sum_{x^{1/4} < p \le \sqrt{\frac{x}{y}}}f(p)\
\sum_{p < q \le y} f(q) \pi_f\left(\frac{x}{pq} \right)
+
\sum_{\sqrt{\frac{x}{y}} < p < x^{1/3}}f(p)\
\sum_{p < q \le \frac{x}{p^2}} f(q) \pi_f\left(\frac{x}{pq} \right)
\]
In the purpose to speed up the computation of the above two sums
we now write,
\dsm{
V_2 = W_1 + W_2 + W_3 + W_4 + W_5
}
with
\begin{eqnarray*}
W_1 &=& 
\sum_{x^{1/4} < p \le \frac{x}{y^2}}f(p)\ \sum_{p < q \le y}\
  f(q)\, \pi_f\left(\frac{x}{pq} \right)\\
W_2 &=& 
\sum_{\frac{x}{y^2} < p \le \sqrt{\frac{x}{y}}} f(p)\ 
\sum_{p < q \le \sqrt{\frac{x}{p}}} 
  f(q)\, \pi_f\left(\frac{x}{pq} \right)\\
W_3 &=& 
\sum_{\frac{x}{y^2} < p \le \sqrt{\frac{x}{y}}}f(p)\ 
\sum_{\sqrt{\frac{x}{p}} < q \le  y} \
  f(q)\, \pi_f\left(\frac{x}{pq} \right)\\
W_4 &=& 
\sum_{\sqrt{\frac{x}{y}} < p \le x^{1/3}}f(p)\ 
\sum_{p  < q \le \sqrt{\frac{x}{p}}} \
  f(q)\, \pi_f\left(\frac{x}{pq} \right)\\
W_5 &=& 
\sum_{\sqrt{\frac{x}{y}} < p \le x^{1/3}}f(p)\ 
\sum_{\sqrt{\frac{x}{p}} <  q \le \frac{x}{p^2}}\
  f(q)\, \pi_f\left(\frac{x}{pq} \right)\\
\end{eqnarray*}

\paragraph{Computing $W_1$ and $W_2$} These two quantities need values
of $\pi_f(x/pq)$ with $x^{1/3} < x/pq < x^{1/2}$. These are computed
with a sieve of the interval $\intf 1 {\sqrt x}$.
The sieving is done by blocks, and, for each block, we sum
$f(p)f(q)\pi_f(x/pq)$ for the pairs $(p,q)$ satisfying the
conditions of the sum $W_1$ or $W_2$ and such that $x/pq$ lies
in the block.

The cost of this computation is the sum of three terms~:
\medskip

\begin{itemize}
\item
The cost of the above sieve on \dsm{\intf{1}{\sqrt x}}
is \dsm{\bigo{\sqrt x \log_2 x}}.
\item
The cost of adding the terms of the sum $W_1$,\ 
\dsm{\bigo{\frac{x}{y\log^2 x}}}.
\item
The cost of adding the terms of the sum $W_2$,
\dsm{\bigo{\frac{x^{3/4}}{y^{1/4}\log^2 x}}}.
\end{itemize}

\paragraph{Computing $W_3$ and $W_5$}
For $W_3$, for each $p$ we apply lemma \ref{speedup} with $z=x/p$ and
$u=y$. Thus, for each value of $p$, the sum on $q$ costs
$\bigo{\pi(\sqrt{x/p})}$, and the total cost of the computation of
$W_3$ is
\begin{equation*}
\bigo{\sum_{\frac{x}{y^2} <
p \le \frac{x}{y}} \pi\left(\sqrt{\frac{x}{p}}\right)}.
\end{equation*}
For $W_5$, for each $p$ we apply lemma \ref{speedup} with $z=x/p$ and
$u=x/p^2$. Thus, for each value of $p$, the sum on $q$ costs
$\bigo{\pi(\sqrt{x/p})}$, and the total cost of the computation of
$W_5$ is
\begin{equation*}
\bigo{\sum_{\sqrt{\frac{x}{y}} < p \le x^{1/3}} \pi\left(\sqrt{\frac{x}{p}}\right)}.
\end{equation*}
Thus the costs of computing $W_3$ and $W_5$ add to
\begin{eqnarray*}\label{w3}\nonumber
\bigo{\sum_{\frac{x}{y^2} < p \le
x^{1/3}} \pi\left(\sqrt{\frac{x}{p}}\right)}.
&=& \bigo{\sum_{\frac{x}{y^2} < p \le
x^{1/3}} \left(\frac{\sqrt{\frac{x}{p}}}{\log{\frac{x}{p}}}\right)}
=
\bigo{\frac{x^{2/3}}{\log^2 x}}\cdot 
\end{eqnarray*}

\paragraph{Computing $W_4$}
We simply sum over $(p,q)$. There would be no advantage to proceed
as for $W_3$ since most of the values $\pi_f(x/pq)$ are distinct.
The cost is
\[
\bigo{\sum_{\sqrt{\frac{x}{y}} < p \le
x^{1/3}} \pi\left(\sqrt{\frac{x}{p}}\right)}
= \bigo{\frac{x^{2/3}}{\log^2 x}}\cdot
\]
\medskip

As in \cite{DRPI}, section 8, we then see that, since $y
=x^{1/3}\log^3 x\log_2 x$, the total cost of the computation of
$\pi_f(x)$ is $\bigo{x^{2/3}/\log^2x}$.

%% file: figure.tex
\renewcommand{\arraystretch}{0.85}
\begin{small}
\begin{figure}[hbt]\label{tab1}
\[
\begin{array}{|r|rrrrrr|}
\hline
 j=&1&2&3&4&5&6\\
\hline
                                    n=2& 2& & & & &\\
                                    3& 3& & & & &\\
                                    4& 3& & & & &\\
                                   5& 5& 6& & & &\\
                                   6& 5& 6& & & &\\
                                  7& 7& 10& & & &\\
                                  8& 7& 15& & & &\\
                                  9& 7& 15& & & &\\
                                10& 7& 21& 30& & &\\
                               11& 11& 21& 30& & &\\
                               12& 11& 35& 42& & &\\
                               13& 13& 35& 42& & &\\
                               14& 13& 35& 70& & &\\
                               15& 13& 35& 105& & & \\
                               16& 13& 55& 105& & &\\
                            17& 17& 55& 105& 210& &\\
                            18& 17& 77& 110& 210& &\\
                            19& 19& 77& 165& 210& &\\
                            20& 19& 91& 165& 210& &\\
                            21& 19& 91& 231& 330& &\\
                            22& 19& 91& 231& 330& &\\
                            23& 23& 91& 385& 462& &\\
                            24& 23& 143& 385& 462& &\\
                            25& 23& 143& 455& 770& &\\
                           26& 23& 143& 455& 1155& &\\
                           27& 23& 143& 455& 1155& &\\
                        28& 23& 187& 455& 1365& 2310& \\
                        29& 29& 187& 715& 1365& 2310& \\
                        30& 29& 221& 715& 1365& 2730& \\
                        31& 31& 221& 1001& 1430& 2730& \\
                        32& 31& 247& 1001& 2145& 2730& \\
                        33& 31& 247& 1001& 2145& 2730& \\
                        34& 31& 253& 1001& 3003& 4290& \\
                        35& 31& 253& 1309& 3003& 4290& \\
                        36& 31& 323& 1309& 5005& 6006& \\
                        37& 37& 323& 1547& 5005& 6006& \\
                       38& 37& 323& 1547& 5005& 10010& \\
                       39& 37& 323& 1729& 5005& 15015& \\
                       40& 37& 391& 1729& 6545& 15015& \\
                    41& 41& 391& 2431& 6545& 15015& 30030\\
                    42& 41& 437& 2431& 7735& 15015& 30030\\
                    43& 43& 437& 2717& 7735& 19635& 30030\\
                    44& 43& 437& 2717& 8645& 19635& 30030\\
                    45& 43& 437& 2717& 8645& 23205& 39270\\
                    46& 43& 493& 2717& 12155& 23205& 39270\\
                    47& 47& 493& 3553& 12155& 25935& 46410\\
                    48& 47& 551& 3553& 17017& 25935& 46410\\
                    49& 47& 551& 4199& 17017& 36465& 51870\\
                    50& 47& 589& 4199& 19019& 36465& 51870\\
\hline
\end{array}
\]
\caption{Table of $h_j(n)$.\label{figtabhj}}
\end{figure}
\end{small}